\documentclass[11pt,a4paper, parskip=half]{scrartcl} 
\usepackage[utf8]{inputenc}
\usepackage[english]{babel}
\usepackage[T1]{fontenc}
\usepackage{lmodern}
\usepackage[left=3cm,right=3cm,top=3cm,bottom=3cm]{geometry}
\usepackage{amsmath, amsfonts, amssymb, amsthm, amstext}
\usepackage{mathtools}
\usepackage{mathrsfs}
\usepackage{mathdots}
\usepackage{cases}
\usepackage{algorithmic}

\font\mfett=cmmib10 at11pt

\def\mubf{\hbox{\mfett\char022}} 
     at9pt
\def\alphabf{\hbox{\mfett\char011}} 
\def\betabf{\hbox{\mfett\char012}} 

\def\sumprime_#1^#2{
    \setbox0=\hbox{$\scriptstyle{#1}$}
    \setbox1=\hbox{$\scriptstyle{#2}$}
    \setbox2=\hbox{$\displaystyle{\sum}$}
    \setbox4=\hbox{${}^\prime\mathsurround=0pt$}
    \dimen0=.5\wd0 \advance\dimen0 by-.5\wd2
    \ifdim\dimen0>0pt
        \ifdim\dimen0>\wd4 \kern\wd4
        \else\kern\dimen0
        \ifdim\dimen1>\wd4 \kern\wd4
        \else\kern\dimen1
    \fi\fi\fi
\mathop{{\sum}^\prime}_{\kern-\wd4 #1}^{\kern-\wd4 #2}
}

\usepackage{hyperref} 
\usepackage[capitalise]{cleveref}
    \crefformat{equation}{\textup{#2(#1)#3}}
    \crefrangeformat{equation}{\textup{#3(#1)#4--#5(#2)#6}}
    \crefmultiformat{equation}{\textup{#2(#1)#3}}{ and \textup{#2(#1)#3}}
    {, \textup{#2(#1)#3}}{, and \textup{#2(#1)#3}}
    \crefrangemultiformat{equation}{\textup{#3(#1)#4--#5(#2)#6}}%
    { and \textup{#3(#1)#4--#5(#2)#6}}{, \textup{#3(#1)#4--#5(#2)#6}}{, and \textup{#3(#1)#4--#5(#2)#6}}

    \Crefformat{equation}{#2Equation~\textup{(#1)}#3}
    \Crefrangeformat{equation}{Equations~\textup{#3(#1)#4--#5(#2)#6}}
    \Crefmultiformat{equation}{Equations~\textup{#2(#1)#3}}{ and \textup{#2(#1)#3}}
    {, \textup{#2(#1)#3}}{, and \textup{#2(#1)#3}}
    \Crefrangemultiformat{equation}{Equations~\textup{#3(#1)#4--#5(#2)#6}}%
    { and \textup{#3(#1)#4--#5(#2)#6}}{, \textup{#3(#1)#4--#5(#2)#6}}{, and \textup{#3(#1)#4--#5(#2)#6}}

    \crefdefaultlabelformat{#2\textup{#1}#3}
	
\newcounter{thm}
\numberwithin{thm}{section}
\numberwithin{equation}{section}

	\newtheoremstyle{myplain}		
			{}			
			{}			
			{\itshape}				
			{}				
			{\sffamily\bfseries}				
			{.}		
			{ }				
			{\thmname{#1}\thmnumber{ #2}\textnormal{\textsf{\thmnote{ (#3)}}}}			
	\newtheoremstyle{mybreak}
            {}{}{}{}{\sffamily\bfseries}{.}{\newline}
            {\thmname{#1}\thmnumber{ #2}\textnormal{\textsf{\thmnote{ (#3)}}}}
	\newtheoremstyle{mydef}
			{}{}{}{}{\sffamily\bfseries}{.}{ }
			{\thmname{#1}\thmnumber{ #2}}
	\newtheoremstyle{myrem}
			{}{}{}{}{\sffamily\itshape}{.}{ }
			{\thmname{#1}\thmnumber{ #2}}

\theoremstyle{myplain}
	\newtheorem{theorem}[thm]{Theorem}
		\crefname{theorem}{Theorem}{Theorems}
    \newtheorem{proposition}[thm]{Proposition}
        \crefname{proposition}{Proposition}{Propositions}
        \Crefname{proposition}{Proposition}{Propositions}
	\newtheorem{lemma}[thm]{Lemma}
		\crefname{lemma}{Lemma}{Lemmas}
		\Crefname{lemma}{Lemma}{Lemmas}
    \newtheorem{corollary}[thm]{Corollary}
\theoremstyle{mybreak}
    \newtheorem{algorithm}[thm]{Algorithm}
\theoremstyle{mydef}
    
    \newtheorem{remark}[thm]{Remark}
\theoremstyle{myrem}
    \newtheorem{example}[thm]{Example}

\DeclareMathOperator{\rank}{rank}

\DeclareMathOperator{\diag}{diag}

\DeclareMathOperator{\vect}{vec}

\DeclarePairedDelimiter{\abs}{\lvert}{\rvert}

\allowdisplaybreaks

\usepackage{todonotes}

\title{\Large Optimal Rank-1 Hankel Approximation of Matrices:\\ Frobenius Norm, Spectral Norm, and Cadzow's Algorithm}
\author{Hanna Knirsch\footnote{Institute for Numerical and Applied Mathematics,n G\"ottingen University, Lotzestr.\ 16-18, 37083 G\"ottingen, Germany, \{h.knirsch,m.petz,plonka\}@math.uni-goettingen.de}, Markus Petz$^{*}$, and Gerlind Plonka$^{*}$ \footnote{Corresponding author}}
\date{}

\begin{document}
	\let\oldproofname=\proofname
	\renewcommand{\proofname}{\itshape\sffamily{\oldproofname}}

\maketitle

\begin{abstract}
\textbf{Abstract.}
We  characterize optimal rank-1 matrix approximations with Hankel or Toeplitz structure with regard to two different norms, the Frobenius norm and the spectral norm,  in a new way.  More precisely, we show that these rank-1 matrix approximation problems can be solved by maximizing special rational functions.
Our approach enables us to show that the optimal solutions with respect to these two norms  have completely different structure and only coincide in the trivial case when the singular value decomposition already provides an optimal rank-1 approximation with the desired Hankel or Toeplitz structure.
  We also prove that the Cadzow algorithm for structured low-rank approximations always converges to a fixed point in the rank-1 case. However, it usually does not converge to the optimal solution, neither with regard to the Frobenius norm nor the spectral norm. 
\end{abstract}

\section{Introduction}
\setcounter{equation}{0}
Structured low-rank approximations are widely used in many signal processing problems as in system theory, parameter identification and signal analysis, e.g.\ singular spectral analysis (SSA) \cite{Golyandina10}. 
Applications include minimal partial realizations in linear system theory, multi-input-multi-output systems, system identification problems or approximation with finite rate of innovation  signals \cite{Fazel13, Liu09, Vetterli.2002}.
Low-rank Hankel approximation is  closely related to Prony's method \cite{PT14},  or related modifications \cite{BM86, Osborne, ZP19}.

Generally, a low-rank Hankel approximation  problem can be written  as a non-convex optimization problem. For a given matrix ${\mathbf A} \in {\mathbb C}^{M \times N}$  one wants to find a Hankel matrix 
$\mathbf{H}_r$ of rank at most $r < \min\{M, N\}$, such that
\begin{equation}\label{prob1}
    \mathbf{H}_{r} \coloneqq \mathop{\mathrm{argmin}}_{\substack{{\mathbf H} \, \text{Hankel}\\ \rank {\mathbf H} \le r}} \| {\mathbf A} - {\mathbf H}\|, 
\end{equation}
where the considered matrix norm is usually taken to be a (weighted) Frobenius norm.

\paragraph{Notation.}
To state the problem precisely, we start with some  notations.
For a given matrix ${\mathbf A}= (a_{j,k})_{j,k=0}^{M-1,N-1} \in {\mathbb C}^{M \times N}$ we define the Frobenius norm and the spectral norm of ${\mathbf A}$ as 
\begin{equation*}
    \| {\mathbf A} \|_{F} \coloneqq \left(\sum_{j=0}^{M-1} \sum_{k=0}^{N-1} |a_{j,k}|^{2} \right)^{1/2} = \text{trace}({\mathbf A}^* \mathbf{A})^{1/2},
    \quad
    \| {\mathbf A} \|_{2} \coloneqq \max_{\|{\mathbf x} \|_{2} = 1} \| {\mathbf A} {\mathbf x} \|_{2} =  \rho({\mathbf A}^{*}  {\mathbf A})^{1/2},
\end{equation*}
where $\| {\mathbf x}\|_{2} \coloneqq (\sum_{j=0}^{N-1} |x_{j}|^{2})^{1/2}$ denotes the Euclidean vector norm. 
Let 
\begin{equation}\label{norm1}
\|{\mathbf A}\|_{\infty} \coloneqq \max\limits_{j,k} |a_{j,k} | \qquad  \textrm{and}  \qquad \|{\mathbf x}\|_{\infty} \coloneqq \max\limits_{j} |x_{j}|.
\end{equation}
Further, ${\mathbf A}^{*} \coloneqq \overline{\mathbf A}^{T}$ and ${\mathbf x}^{*} \coloneqq \overline{\mathbf x}^{T}$ denote the complex conjugate and transpose of a matrix or vector, respectively. By $\rho({\mathbf A}^{*}  {\mathbf A})$ we denote the spectral radius of the positive semi-definite matrix ${\mathbf A}^{*}  {\mathbf A}$, i.e., the largest eigenvalue of ${\mathbf A}^{*}  {\mathbf A}$. 
The singular value decomposition of ${\mathbf A}$ is given by 
$\mathbf A = \mathbf{U\Sigma V}^*$, where ${\mathbf U} \in {\mathbb C}^{M\times M}$ and ${\mathbf V} \in {\mathbb C}^{N \times M}$ satisfy ${\mathbf U}^{*} {\mathbf U} = {\mathbf V}^{*} {\mathbf V} = {\mathbf I}_{M}$ (with the $M \times M$-identity matrix ${\mathbf I}_{M}$), and $\mathbf{\Sigma} \coloneqq \textrm{diag}(\sigma_{0}, \ldots , \sigma_{M-1}) \in {\mathbb R}^{M \times M}$ with the ordered singular values $\sigma_0 \geq \sigma_1 \geq \dots \geq \sigma_{M-1}$. If $\mathbf A \in {\mathbb R}^{N \times N}$ admits an eigendecomposition, we write $\mathbf A = \mathbf{V\Lambda V}^T$, where ${\mathbf V} \in {\mathbb R}^{N \times N}$ is orthogonal, $\mathbf\Lambda = \diag(\lambda_0,\dots,\lambda_{M-1})$, and the eigenvalues $|\lambda_0| \geq |\lambda_1| \geq \dots \geq |\lambda_{N-1}|$ are ordered by modulus.

Hankel matrices are of the form 
\begin{equation}\label{hankel}
    {\mathbf H} \coloneqq (h_{k+\ell})_{k,\ell=0}^{M-1, N-1} = 
    \begin{pmatrix}
        h_0 & h_1 & h_2 & \cdots & h_{N-1} \\
        h_1 & h_2 &  &  & h_{N} \\
        h_2 &     &       &  & \vdots \\
        \vdots & & & & \vdots \\
        h_{M-1} & h_{M} & h_{M+1} & \cdots & h_{M+N-2}
    \end{pmatrix} \in \mathbb{C} ^{M \times N}.
\end{equation}
We denote the orthogonal projection of a general matrix $\mathbf A \in {\mathbb C}^{M \times N}$ with $M \le N$, onto the linear space of Hankel structured matrices by $P(\mathbf A)$. It is obtained by averaging the matrix elements along counter diagonals, i.e.,
\begin{align}
    \label{eq:H-proj-1}
    P({\mathbf A}) &\coloneqq (h_{k+\ell})_{k, \ell=0}^{M-1, N-1} \in {\mathbb C}^{M \times N} \\
\shortintertext{with}
    \label{eq:H-proj-2}
    h_\ell &\coloneqq \begin{cases}
                        \frac{1}{\ell+1} \sum\limits_{r=0}^{\ell} a_{r,\ell-r}  & \text{for} \quad \ell=0, \ldots , M-1, \\
                        \frac{1}{M} \sum\limits_{r=0}^{M-1} a_{r,\ell-r}  & \text{for} \quad \ell=M, \ldots, N-1,\\
                        \frac{1}{M+N-1-\ell} \sum\limits_{r=\ell+1-N}^{M-1} a_{r,\ell-r}  & \text{for} \quad \ell=N, \ldots , N+M-2.
                     \end{cases} 
\end{align}
For $M > N$ we take $P({\mathbf A}) \coloneqq P({\mathbf A}^{T})^{T}$, see also \cref{sec:cad}.

Let ${\mathbf I}_{N} \in {\mathbb R}^{N \times N}$ denote the identity matrix, and the  counteridentity matrix is given by 
\begin{equation}\label{J}
    {\mathbf J}_{N} \coloneqq \begin{pmatrix}
                                0 & \ldots & 0 & 1\\
                                \vdots &  & 1 & 0 \\
                                0 &  & & \vdots \\
                                1 & 0 & \ldots & 0
                            \end{pmatrix} \in {\mathbb C}^{N \times N}.
\end{equation}
We introduce the normalized structured vector
\begin{equation}\label{z}
    {\mathbf z}_{N}(z) = {\mathbf z}_{N} \coloneqq \left(\sum\limits_{k=0}^{N-1} |z|^{2k}\right)^{-1/2}  \left( z^{k} \right)_{k=0}^{N-1}
    = \left(\sum\limits_{k=0}^{N-1} |z|^{2k}\right)^{-1/2} \left( 1, z, z^{2}, \ldots, z^{N-1} \right)^{T},
\end{equation}
for any number $z\in {\mathbb C}$ and $N\in\mathbb N$.
We use the convention that ${\mathbf z}_{N}(z)$ is abbreviated by ${\mathbf z}_{N}$ or just by ${\mathbf z}$, if the dimension and the argument $z$ is clear from the context.
Furthermore, let
\begin{equation}\label{en}
    \tilde{\mathbf e}_{N} \coloneqq (1, 0 \ldots ,  0)^{T} \in {\mathbb C}^{N}
    \qquad \text{and} \qquad 
    {\mathbf e}_{N} \coloneqq (0, \ldots , 0, 1)^{T} \in {\mathbb C}^{N}
\end{equation}
be the first and last vector of the standard basis, respectively.
Finally, note that in \cref{sec:spec} we will use the notation $\sumprime_{}^{}$ instead of $\sum$ if terms of the form $\frac{0}{0}$ appear in this sum, and in $\sumprime_{}^{}$ such terms are just omitted.

\paragraph{Statement of the problem.}
In this paper, we are  interested in optimal approximations of a given matrix ${\mathbf A} \in \mathbb{C}^{M \times N}$ by a rank-1 Hankel matrix ${\mathbf H}_{1}$ of the same size
with regard to the Frobenius norm, i.e., we want to solve
\begin{equation}\label{problem}
	\min_{\mathbf{H}_{1}\in\mathbb{C}^{M\times N}} \lVert{\mathbf A}- {\mathbf H}_{1}\rVert_F \end{equation}
under the restriction that $\mathbf{H}_{1}$ is a Hankel matrix of rank $1$.
Further, we consider for symmetric matrices ${\mathbf A} \in {\mathbb R}^{N \times N}$ the rank-1 approximation problem with regard to the spectral norm
\begin{equation}\label{problem1}
 \min_{\mathbf{H}_{1}\in\mathbb{R}^{N\times N}} \lVert{\mathbf A}-  \mathbf{H}_{1}\rVert_{2}, 
\end{equation}
under the restriction that $\mathbf{H}_{1}$ is a Hankel matrix of rank $1$.
It is well-known that the  (unstructured) rank-1 approximation problem can be directly solved using the singular value decomposition (SVD) of ${\mathbf A}$, but this SVD-approximation usually does no longer possess the wanted Hankel structure. The minimization problems \cref{problem} and \cref{problem1} are non-convex and in particular for the spectral norm of highly nontrivial structure. 
While we will always consider Hankel matrices in this paper, we remark that the minimization problems in \cref{problem} can be rewritten using Toeplitz matrices instead of Hankel matrices. Since a Toeplitz matrix ${\mathbf T} = (h_{k-l})_{k,l=0}^{M-1,N-1} \in \mathbb{C}^{M\times N}$ can be represented as
\begin{equation*}
    {\mathbf T} = {\mathbf H} \, {\mathbf J}_{N}
\end{equation*}
with ${\mathbf H}$ in  \cref{hankel} and ${\mathbf J}_{N}$ in \cref{J}, we obtain
\begin{equation*}
	\min_{{\mathbf T}_{1}\in\mathbb{C}^{M\times N}} \lVert{\mathbf A}- {\mathbf T}_{1}\rVert^2_F = \min_{\mathbf{H}_{1}\in\mathbb{C}^{M\times N}} \lVert{\mathbf A}- {\mathbf H}_{1} {\mathbf J}_{N}\rVert^2_F = \min_{\mathbf{H}_{1}\in\mathbb{C}^{M\times N}} \lVert{\mathbf A}{\mathbf J}_{N} - {\mathbf H}_{1}\rVert^2_F,
\end{equation*}
where ${\mathbf T}_{1}$ denotes a Toeplitz matrix of rank $1$.
This transfer works likewise for the spectral norm.

\paragraph{Main results.}
In this paper, we analytically reformulate the rank-1 Hankel approximation problem such that numerical computation of the optimal solutions for \cref{problem,problem1}  becomes feasible. For the Frobenius norm, the optimal rank-1 approximation problem for matrices ${\mathbf A} \in {\mathbb C}^{M \times N}$ can be restated as a maximization problem for a rational function.
The main results to solve \cref{problem} are stated in \cref{theo1,theoneu}.
These results can be simply transferred to a weighted Frobenius norm, see \cref{rem3.2}.
 For the spectral norm, the problem is much more delicate and of different nature. Therefore, we have to restrict ourselves to symmetric matrices ${\mathbf A} \in {\mathbb R}^{N \times N}$. In this case, we can characterize the optimal rank-1 Hankel approximation of ${\mathbf A}$ as a maximization problem for a rational function that also depends on the optimal approximation error.
Our main result to solve \cref{problem1} is \cref{theospectral1}.
Our results give rise to corresponding algorithms to derive the optimal rank-1 Hankel approximations numerically. The obtained characterizations of optimal solutions for \cref{problem,problem1} can therefore serve as benchmarks for comparison of other optimization  approaches for low-rank Hankel approximation.
Because of the completely different structure of the optimal solutions in the two considered norms, we usually get different optimal rank-1 approximations.
Moreover, in \cref{theo2,theospectral1}, we provide necessary and sufficient conditions on ${\mathbf A}$ ensuring that the optimal rank-1 Hankel approximation error for the solutions of \cref{problem} coincides with the error achieved by the best unstructured rank-1 approximation.\\
Further, we present a complete proof that Cadzow's algorithm always converges to a fixed point in the rank-1 Hankel approximation case. 
As far as we know, this is the first  complete convergence proof for the Cadzow algorithm in the considered special case despite partial convergence results, see e.g.\ \cite{Zvonarev17}. General results on convergence of  alternating projection algorithms are not simply applicable in the case of low-rank Hankel approximation, see \cite{Andersson13,Condat15,Lewis}.
However, (up to trivial cases) this fixed point does not coincide with the optimal solution of \cref{problem},  neither for the Frobenius nor for the spectral norm.
This result confirms previous results on the behaviour and convergence of alternating projection algorithms for the considered special case of rank-1-Hankel approximation, see e.g.\ \cite{Andersson13,deMoor93}.

\paragraph{Related approaches.}
There are different optimization approaches  in the literature  to tackle the structured low-rank approximation problem, which all focus on the (weighted) Frobenius norm.
One heuristic approach, often used in practice because of its simplicity, is Cadzow's algorithm \cite{Cadzow88,Chu03,Andersson13}, which is an alternating projection method, see also \cref{sec:cad}.\\
In case of the (weighted) Frobenius norm, problem \cref{prob1} can also be written as a non-linear eigenvalue problem, see \cite{BM86, Osborne, ZP19}, or as a non-linear structured least squares problem (NSLSP), see e.g.\ \cite{deMoor93, deMoor94, Gillard11, Ishteva14, Lemmerling00, Lemmerling01, Markovsky18, Markovsky05, UM14, Zvonarev18}. When applying the NSLSP methods one usually assumes that the initial matrix ${\mathbf A}$ itself is already structured (here Hankel or Toeplitz). This is not a limitation, though. A general matrix $\mathbf A$ can be first projected onto the subspace of structured matrices, i.e., one can employ $P({\mathbf A})$ in \cref{eq:H-proj-1} instead of ${\mathbf A}$, see also \cref{rem3.2}. \\
Some methods  are based on relaxation  of the optimization problem  using the nuclear norm \cite{Fazel13}, convex envelopes \cite{Andersson19, Grussler18} or subspace based  and hybrid methods \cite{Liu09, VanOver96}.\\
In \cite{Otta14} and \cite{Usevich2012}, a reformulation of the structured low-rank approximation problem is considered, which is restricted to real matrices and to a weighted Frobenius norm.
However, our results for solving \cref{problem} differ from the analytical characterization of the solution that can be obtained by applying the methods of \cite{Otta14} and \cite{Usevich2012} to the rank-1 Hankel approximation. 
In particular, our characterization in \cref{theo3} leads to a zero set of a polynomial with much smaller degree than the approach of \cite{Otta14}.  
A completely different idea to study the structured low-rank approximation problem arises from the AAK theory \cite{aak} for optimal low-rank approximation of Hankel operators.  The AAK theory shows  that infinite Hankel matrices (with certain decay properties of their components) can always  be approximated  by infinite Hankel matrices of lower rank with optimal error. This means, similarly as in the case of unstructured matrices, the (operator norm) error of the rank-$r$ approximation is given by the $(r+1)$-st largest singular value of the Hankel operator. These optimal infinite low-rank Hankel matrices can also be computed numerically, see \cite{Beylkin05, PP16}, and have been used to compute adaptive Fourier series with exponential decay for large function classes in \cite{PP19}.
Unfortunately, this approach cannot be directly transferred to finite matrices, \cite{Beylkin05}.\\
Regarding the structured low-rank approximation problem for finite matrices in the spectral norm  there are almost no previous results in the literature.
For regular real $(N\times N)$-matrices, the minimal distance to a singular (i.e., rank-deficient) structured matrix has been studied both for the Frobenius and the spectral norm, see \cite{Rump2003, Rump2003a}. Lemma 10.1 in \cite{Rump2003} (and similarly Lemma 4.7 in \cite{PP16} for the case of infinite Hankel matrices of finite rank) can indeed be exploited to construct an optimal rank-($N-1$) Hankel approximation for a given Hankel matrix of rank $N$, which is optimal with regard to the spectral norm. Unfortunately, this approach cannot be extended to construct  optimal Hankel approximations of lower rank.  
In \cite{Antoulas97}, Antoulas studied necessary and sufficient conditions for achieving an error of the rank-1 Hankel approximation, which is as good as for the unstructured case in the spectral norm. He also restricted his considerations to real symmetric Hankel matrices ${\mathbf A}$, compare \cref{theospectral1}.\\
Finally, we want to mention that the problem of finding extreme values of a special rational function, as derived in our \cref{theo1,theoneu}, also appears in other contexts as e.g.\ for the problem of computing the GCDs of univariate polynomials, see \cite{Cheze2011,Karmarkar1998}, and one may therefore apply similar strategies to solve this optimization problem.

\paragraph{Organization of this paper.}
In \cref{sec:hankel_mat} we recall that Hankel matrices of rank $1$ have a special structure and can be determined by two complex parameters $c$ and $z$.\\
In \cref{sec:frob} we show how the optimal rank-1 approximation with regard to the Frobenius norm can be obtained.
The main theoretical results are stated in \cref{theo1,theo2} in the complex case.
In \cref{theo2} we provide necessary and sufficient conditions on ${\mathbf A}$ ensuring that the optimal rank-1 approximation error coincides  with the error achieved for unstructured rank-1 approximation.
Further, we present a series of results that simplify the computation of the optimal rank-1 Hankel approximation in the real case.\\
In \cref{sec:spec}, we solve the rank-1 Hankel approximation problem with regard to the spectral norm for symmetric square matrices.
The result also gives rise to a corresponding algorithm. The completely different structure of optimal rank-1 Hankel approximations for the Frobenius norm and the spectral norm implies that the results usually differ.\\
\Cref{sec:cad} is devoted to Cadzow's algorithm, which is (despite a lot of existing optimization approaches) the most popular method for low-rank Hankel approximation in practice. We give a new direct proof, that the Cadzow algorithm always converges to a fixed point in the rank-1 case. However, we observe in our numerical examples that it usually does not converge to the optimal solution, neither with regard to the Frobenius norm nor the spectral norm. It may even fail completely, as \cref{Cadex} shows.

\section{Rank-1 Hankel Matrices}
\label{sec:hankel_mat}
\setcounter{equation}{0}

Our approach to find the optimal Hankel-structured rank-1 approximation of a given matrix ${\mathbf A} \in {\mathbb C}^{M\times N}$ will be based on the following canonical characterization of rank-1 Hankel matrices, which is a special case of the results in \cite{Heinig1984}, Theorem 8.1.

\begin{lemma}\label{thm:rank-1_hankel}
A complex  rank-1 matrix ${\mathbf H}_{1} \in {\mathbb C}^{M \times N}$ with $\min\{M, \, N \} \ge 2$ has Hankel structure if and only if it is of the form
\begin{equation}\label{hh}
	\mathbf{H}_1 = c \, {\mathbf z}_{M} \, {\mathbf z}_{N}^{T} 
      \qquad       	\text{or} \qquad	\mathbf{H}_1 = c \, {\mathbf e}_{M} \, {\mathbf e}_{N}^{T}  = \begin{pmatrix} {\mathbf 0} & {\mathbf  0} \\ {\mathbf 0} & c \end{pmatrix},
\end{equation}
where $c\in\mathbb{C}\setminus\{0\}$, $z\in {\mathbb C}$, and $\mathbf{z}_{N}$, $\mathbf{e}_N$ as in \cref{z,en}, respectively.
\end{lemma}

\begin{proof}
We give a short proof for  the convenience of the reader.
Obviously, the two matrices $\mathbf{H}_1 = c \, {\mathbf z}_{M} \, {\mathbf z}_{N}^{T}$ and $\mathbf{H}_1 = c \, {\mathbf e}_{M} \, {\mathbf e}_{N}^{T}$ are rank-1 matrices with Hankel structure.
 
We show that each rank-1 Hankel matrix $\mathbf{H}_{1} = (h_{k+\ell})_{k,\ell=0}^{M-1,N-1}$ of rank 1 has the desired form \cref{hh}.
Since $\mathbf{H}_{1}$ has rank 1, we obtain the representation
  $ \mathbf{H}_{1}  = {\mathbf x} \, {\mathbf y}^{T}$
 for some vectors  ${\mathbf x} = (x_{0}, \ldots, x_{M-1})^{T} \in {\mathbb C}^{M}$ and ${\mathbf y}=(y_{0}, \ldots , y_{N-1})^{T} \in {\mathbb C}^{N}$. The imposed Hankel structure implies  the conditions
 \begin{equation}\label{index}
     x_k y_\ell = x_m y_n, \quad \text{for } k+\ell = m+n,
 \end{equation}
 where $k, m = 0,\dots,M-1$ and $\ell,n = 0,\dots,N-1$.
Assume first that $h_{0}= x_{0} y_{0} \neq 0$ and define $z\coloneqq x_{1}/x_{0}$.  
 It follows from \cref{index} with $k+\ell = 1$, i.e., from $x_0 y_1 = x_1 y_0$, that $z = y_1/y_0$, and thus
    $ x_1 = z x_0$ and $y_1 = z y_0$.
Using \cref{index}, it can be shown by induction that $x_{j}= z^{j} \, x_{0}$ for $j=1, \ldots , M-1$, and $y_{j}= z^{j} \, y_{0}$ for $j=1, \ldots , N-1$.
Thus, $\mathbf{H}_{1}$ has the desired structure $c \, {\mathbf z}_{M} \, {\mathbf z}_{N}^{T}$ with $z=x_{1}/x_{0}$ and $c=x_{0} \, y_{0} \| (z^{k})_{k=0}^{M-1} \|_{2} \| (z^{k})_{k=0}^{N-1} \|_{2}$.
If $h_{0}=x_{0} \, y_{0} = 0$ then either $x_{0}=0$ or $y_{0}=0$. 
Thus, either the complete first row or the complete first column of $\mathbf H_{1}= {\mathbf x} {\mathbf y}^{T}$ contains only zeros.
By obeying the Hankel structure and the rank-1 condition, we inductively obtain that 
${\mathbf H}_{1} = c \, {\mathbf e}_{M} {\mathbf e}_{N}^{T}$.
\end{proof}

\begin{remark} \label{rem:singval}
    The  matrix $\mathbf{H}_{1} = c \, {\mathbf z}_{M} \, {\mathbf z}_{N}^{T}$ possesses the non-zero singular value $|c|$
    with corresponding left and right singular vectors $\mathbf{z}_M$ and $\mathbf{\bar{z}}_N$, respectively.
\end{remark}

\begin{remark}\label{rem1}
We observe that ${\mathbf e}_{N} = \lim_{z \to \infty} {\mathbf z}_{N}$, where we understand $z\to\infty$ as $|z|\to\infty$ and $\text{Im}(z)\to 0$. Then the special case $\mathbf{H}_1 = c \, {\mathbf e}_{M} \, {\mathbf e}_{N}^{T}$ in \cref{thm:rank-1_hankel}  can also be understood as the limit case  for $z \to \infty$. A similar notation for $z= \infty$  has been also used in \cite{Heinig1984}.
If we define 
\begin{equation}\label{w}
{\mathbf w}_{N}(z)= {\mathbf w}_{N}\coloneqq {\mathbf J}_{N} {\mathbf z}_{N} \in {\mathbb C}^{N} 
\end{equation}
with  ${\mathbf z}_{N}$ in \cref{z}, we can show analogously to \cref{thm:rank-1_hankel} that a rank-1 Hankel matrix ${\mathbf H}_{1} \in {\mathbb C}^{M \times N}$ is of the form 
\begin{equation*}
\mathbf{H}_{1} = c \, {\mathbf w}_{M} \, {\mathbf w}_{N}^{T} \qquad \text{or} \qquad \mathbf{H}_{1} = c \, \tilde{\mathbf e}_{M} \tilde{\mathbf e}_{N}^{T}
\end{equation*}
with $\tilde{\mathbf e}_{N} = \lim\limits_{z \to \infty} {\mathbf w}_{N}$ in (\ref{en}).
\end{remark}

\section{Optimal Rank-1 Hankel Approximation in the Frobenius Norm}
\label{sec:frob}
\setcounter{equation}{0}

\subsection{Complex Rank-1 Hankel Approximations}

First we consider the minimization problem \cref{problem} in the Frobenius norm for complex matrices. We can assume that ${\mathbf A} =(a_{j,k})_{j,k=0}^{M-1,N-1}\in {\mathbb C}^{M \times N}$ satisfies  $|a_{0,0}| \ge |a_{M-1,N-1}|$, otherwise ${\mathbf A}$ may be simply replaced by 
${\mathbf J}_M {\mathbf A}{\mathbf J}_N$.
The assumption $|a_{0,0}| \ge |a_{M-1,N-1}|$ implies that 
$\left\lVert{\mathbf A}- c\, {\mathbf e}_{M} \, {\mathbf e}_{N}^{T}\right\rVert^2_F \ge \left\lVert{\mathbf A}- c\, {\mathbf z}_{M}(0) \, {\mathbf z}_{N}(0)^{T}\right\rVert^2_F$ such that the special case
${\mathbf H}_{1} = c \, {\mathbf e}_{M}{\mathbf e}_{N}^{T}$ will not occur as (the only) desired optimal rank-1 Hankel approximation of ${\mathbf A}$ and can be dropped.
\cref{thm:rank-1_hankel} implies that the minimization problem (\ref{problem}) can be reformulated as 
\begin{equation}\label{eq:minprob_1}
	\min_{c,z\in\mathbb{C}} \left\lVert{\mathbf A}- c\, {\mathbf z}_{M} \, {\mathbf z}_{N}^{T}\right\rVert^2_F, 
\end{equation}
i.e., we only need to find the two constants $c \in {\mathbb C} \setminus \{ 0 \}$ and $z \in {\mathbb C}$, such that the error $ {\mathbf A}- c\, {\mathbf z}_{M} \, {\mathbf z}_{N}^{T}$ is minimized in the Frobenius norm.

\begin{theorem}\label{theo1}
Let ${\mathbf A} =(a_{j,k})_{j,k=0}^{M-1,N-1}\in {\mathbb C}^{M \times N}$ with $M,N \ge 2$ and $|a_{0,0}| \ge |a_{M-1,N-1}|$. Assume that $\rank({\mathbf A}) \ge 1$.
Then an optimal rank-1 Hankel approximation $\mathbf{H}_{1} = \tilde{c} \, \tilde{\mathbf z}_{M} \, \tilde{\mathbf z}_{N}^{T}$ of ${\mathbf A}$ is determined by
\begin{equation}\label{Frobenius}
    \tilde{z} \in  \mathop{\mathrm{argmax}}\limits_{z \in {\mathbb C}}  |{\mathbf z}_{M}^{*} {\mathbf A} \overline{\mathbf z}_{N}| \qquad 
    \tilde{c} \coloneqq \tilde{\mathbf z}_{M}^{*} {\mathbf A} \overline{\tilde{\mathbf z}}_{N},
\end{equation}
where the vectors $\tilde{\mathbf z}_{M}$ and $\tilde{\mathbf z}_{N}$ are defined by $\tilde{z}$ via \cref{z} and $\tilde{\mathbf z}^* \coloneqq \overline{\tilde{\mathbf z}}^{T}$.
\end{theorem}

\begin{proof}
Using the definition of the Frobenius norm, we obtain
\begin{align}\nonumber
    \left\| {\mathbf A} - c \, {\mathbf z}_{M} \, {\mathbf z}_{N}^{T} \right\|_{F}^{2}
    &= \textrm{trace} \, \Big( ({\mathbf A} - c \, {\mathbf z}_{M} \, {\mathbf z}_{N}^{T})^{*} ({\mathbf A} - c \, {\mathbf z}_{M} \, {\mathbf z}_{N}^{T}) \Big) \\
    \nonumber 
    &= \textrm{trace} \, \Big( {\mathbf A}^{*} {\mathbf A}  - c\,  {\mathbf A}^{*}{\mathbf z}_{M} \, {\mathbf z}_{N}^{T} - \overline{c} \, \overline{{\mathbf z}}_{N}\,  {\mathbf z}_{M}^{*} \, {\mathbf A} + |c|^{2} \, \overline{{\mathbf z}}_{N} \, {\mathbf z}_{M}^{*} {\mathbf z}_{M} \, {\mathbf z}_{N}^{T} \Big)\\
\label{frob1}
    &= \| {\mathbf A} \|_{F}^{2} - c \, {\mathbf z}_{M}^{T} \, \overline{\mathbf A} \, {\mathbf z}_{N} - \overline{c} \, {\mathbf z}_{M}^{*} {\mathbf A} \, \overline{\mathbf z}_{N} + |c|^{2},
\end{align}
where ${\mathbf z}_{M}^{*} \coloneqq \overline{\mathbf z}_{M}^{T}$. To solve the minimization problem in \cref{eq:minprob_1}, we first assume $z$ to be fixed and consider the derivatives with respect to $c_1$ and $c_2$, where $c = c_{1} + {\mathrm i} c_{2}$  with $c_{1}, c_{2} \in {\mathbb R}$ to obtain the necessary conditions
\begin{align*}
    \frac{\partial }{\partial c_{1}} \left\| {\mathbf A} - c \, {\mathbf z}_{M} \, {\mathbf z}_{N}^{T} \right\|_{F}^{2} &= -2 \, \text{Re} ( {\mathbf z}_{M}^{*} {\mathbf A} \overline{\mathbf z}_{N}) + 2 c_{1} = 0, \\
    \frac{\partial }{\partial c_{2}} \left\| {\mathbf A} - c \, {\mathbf z}_{M} \, {\mathbf z}_{N}^{T} \right\|_{F}^{2} &= -2 \, \text{Im} ( {\mathbf z}_{M}^{*} {\mathbf A} \overline{\mathbf z}_{N}) + 2 \, c_{2}  = 0.
\end{align*}
These yield the optimal $\tilde{c}= {\mathbf z}_{M}^{*} {\mathbf A} \overline{\mathbf z}_{N}$.
After substituting  $\tilde{c}$ into \cref{frob1} it remains to solve
$$
     \min_{z \in {\mathbb C}} \left( \| {\mathbf A} \|_{F}^{2} - 2 |{\mathbf z}_{M}^{*} {\mathbf A} \overline{\mathbf z}_{N}|^{2}    + |{\mathbf z}_{M}^{*} {\mathbf A} \overline{\mathbf z}_{N}|^{2} \right) 
=   \min_{z \in {\mathbb C}} \left(\| {\mathbf A} \|_{F}^{2} - |{\mathbf z}_{M}^{*} {\mathbf A} \overline{\mathbf z}_{N}|^{2} \right).
$$
Therefore,
\begin{equation*}
    \tilde{z} \in  \mathop{\mathrm{argmax}}_{z \in {\mathbb C}}  |{\mathbf z}_{M}^{*} {\mathbf A} \overline{\mathbf z}_{N}|^{2}
    = \mathop{\mathrm{argmax}}_{z \in {\mathbb C}}   |{\mathbf z}_{M}^{T} \overline{\mathbf A} {\mathbf z}_{N}|
\end{equation*}
as claimed.
\end{proof}

\begin{remark} \label{rem3.2}
1.\ By \cref{theo1}, the computation of the optimal rank-1 Hankel approximation of the matrix ${\mathbf A}$ reduces to the problem of finding a position, where the maximum of the  complex rational function  $|F(z)|^{2}$
with
\begin{equation}\label{fz}
    F(z) \coloneqq {\mathbf z}_{M}^{T} \overline{\mathbf A} {\mathbf z}_{N}
\end{equation}
is attained.
According to \cref{theo1}, we obtain $\tilde{z} = \mathop{\mathrm{argmax}}_{z} |F(z)|$ and $\tilde{c}= F(\tilde{z}).$
Since $\| {\mathbf z}_{M}\|_{2} = 1$ for all $z \in {\mathbb C}$, the function $F(z)$ has no poles. Moreover, $|F(z)|$ is bounded by $\|{\mathbf A}\|_{2}$, which follows from the proof of the next theorem.
If additionally ${\mathbf A} \in {\mathbb R}^{N\times N}$ is symmetric or ${\mathbf A} \in {\mathbb C}^{N\times N}$ is Hermitian, then $F(z)$ is a Rayleigh quotient and thus $\lambda_{min} \leq F(z) \leq \lambda_{max}$, where $\lambda_{min}$ and $\lambda_{max}$ are the smallest and largest eigenvalue of $\mathbf A$, respectively,
see e.g. \cite{Horn85} p.\ 176.

2.\ The value $\tilde{z}$ in \cref{Frobenius} may not be unique, i.e., $\max_{z \in {\mathbb C}} |F(z)|$ may be attained  for different values $\tilde{z}$. In this case, any of these values leads to an optimal Hankel rank-1 approximation. If for example ${\mathbf A} = (a_{j,k})_{j,k = 0}^{M-1, N-1} \in {\mathbb C}^{M \times N}$ is itself a Hankel matrix where $a_{j,k} = 0$ if $j+k$ is odd, 
then $\tilde{z}  \in \mathop{\mathrm{argmax}}_{z \in {\mathbb C}} |F(z)|$ implies that also $-\tilde{z} \in \mathop{\mathrm{argmax}}_{z \in {\mathbb C}} |F(z)|$.

3.\ For the function  $F(z)$ in \cref{fz} we observe that 
\begin{equation*}
    F(z) = {\mathbf z}_{M}^{T} \overline{\mathbf A} {\mathbf z}_{N}
    = \frac{\sum\limits_{\ell=0}^{M+N-2} 
    \left(\sum\limits_{j+k=\ell} \overline{a}_{j,k}\right) \, z^{\ell}}{ \left(\sum\limits_{j=0}^{M-1} |z|^{2j}\right)^{1/2} \, \left(\sum\limits_{j=0}^{N-1} |z|^{2j}\right)^{1/2}} 
    = {\mathbf z}_{M}^{T} \overline{P({\mathbf A})} {\mathbf z}_{N}
\end{equation*}
for ${\mathbf A}=(a_{j,k})_{j,k=0}^{M-1,N-1}$. 
Therefore, without loss of generality, ${\mathbf A}$ can be replaced by the Hankel matrix ${\mathbf H} = P({\mathbf A})$  in (\ref{eq:H-proj-1}). 
Further, the numerator  $\sum\limits_{\ell=0}^{M+N-2} 
\Big(\sum\limits_{j+k=\ell} \overline{a}_{j,k}\Big) \, z^{\ell}$ does not depend on reshaping of the matrix $P({\mathbf A})$
 to a $1 \times (M+N-2)$ Hankel matrix generated by $h_{\ell}:=\sum\limits_{j+k=\ell} {a}_{j,k}$, $\ell=0, \ldots, M+N-2$,   see also \cite{Heinig1984}.
This observation gives the link to the rational function approach in \cite{Usevich2012}.
However, while such reshaping does not change the rank of the Hankel matrix, it changes the solution of (\ref{problem}) since the denominator $\Big(\sum\limits_{j=0}^{M-1} |z|^{2j}\Big)^{1/2} \, \Big(\sum\limits_{j=0}^{N-1} |z|^{2j}\Big)^{1/2}$ strongly depends on the shape of the Hankel matrix.

4.\ In \cite{Karmarkar1998,Cheze2011}, a similar rational approximation problem  as \cref{Frobenius} appears in the context of finding an approximate greatest common divisor. In that context, one needs to minimize a rational function, where the denominator has exactly the same structure if $M=N$. In \cite{Cheze2011}, a subdivision method on squares in the complex plane is proposed to solve that problem.

\end{remark}

We may ask, how well a matrix ${\mathbf A}$ can be approximated by a rank-1 Hankel matrix $\mathbf{H}_{1}$. More precisely, we ask in which cases the Hankel-structured rank-1 approximation is as good as the unstructured rank-1 approximation. The unstructured low-rank approximation is given by the singular value decomposition according to the Eckart-Young-Mirsky Theorem.
Let  ${\mathbf u}_0 \in \mathbb C^M$ and ${\mathbf v}_0 \in \mathbb C^N$ denote the normalized singular vectors corresponding to the largest singular value $\sigma_0 = \lVert \mathbf A \rVert_2$ of $\mathbf A$. Then ${\mathbf u}_0$ and ${\mathbf v}_0$ are determined by the following set of equations
\begin{equation}\label{sing}
    {\mathbf A} {\mathbf A}^{*} {\mathbf u}_0 = \sigma_{0}^{2} {\mathbf u}_0, \quad 
    {\mathbf A}^{*} {\mathbf A} {\mathbf v}_0 = \sigma_{0}^{2} {\mathbf v}_0, \quad
    \textrm{and} \quad 
    {\mathbf u}_0= \frac{1}{\sigma_{0}} {\mathbf A} {\mathbf v}_0, \quad
    {\mathbf v}_0 = \frac{1}{\sigma_{0}} {\mathbf A}^{*} {\mathbf u}_0.
\end{equation} 
We show that the optimal approximation error can only be achieved if the singular vectors ${\mathbf u}_0$ and ${\mathbf v}_0$ corresponding to the largest singular value $\sigma_0$ have the special structure $\mathbf z_M$ and $\mathbf z_N$, respectively, for some $z\in\mathbb C$.

\begin{theorem}\label{theo2}
Let ${\mathbf A} =(a_{j,k})_{j,k=0}^{M-1,N-1}\in {\mathbb C}^{M \times N}$ with $M,N \ge 2$ and $|a_{0,0}| \ge |a_{M-1,N-1}|$.
The optimal rank-1 Hankel approximation error satisfies
\begin{equation}\label{lowapp}
 \min_{c,z \in {\mathbb C}} \left\| {\mathbf A} - {c} \, {\mathbf z}_{M}{\mathbf z}_{N}^{T} \right\|^{2}_{F}  = \left\| {\mathbf A} - \tilde{c} \, \tilde{\mathbf z}_{M} \tilde{\mathbf z}_{N}^{T} \right\|^{2}_{F}
    = \| {\mathbf A} \|_{F}^{2} - \sigma_{0}^{2} = \| {\mathbf A} \|_{F}^{2} - \| {\mathbf A} \|_{2}^{2},
\end{equation}
if and only if the two singular vectors of ${\mathbf A}$ in \cref{sing} corresponding to the largest singular value $\sigma_{0}$ are of the form ${\mathbf u}_0=  \tilde{\mathbf z}_{M}$ and ${\mathbf v}_0 =\overline{\tilde{\mathbf z}}_{N}$, where $\tilde{\mathbf z}_{M}$, $\tilde{\mathbf z}_{N}$ are defined by $\tilde{z}$ via \cref{z},  and where $\tilde{z}$ and $\tilde{c}$ are given by \cref{Frobenius}.
\end{theorem}

\begin{proof}
1.\ Considering the singular value decomposition of ${\mathbf A}$ we obtain an optimal (unstructured) rank-1 approximation of ${\mathbf A}$ with respect to the Frobenius norm of the form $\sigma_{0} \, {\mathbf u}_0 \, {\mathbf v}_0^{*}$ with ${\mathbf u}_0$, ${\mathbf v}_0$ in \cref{sing}. If now ${\mathbf u}_0 =  \tilde{\mathbf z}_{M}$ and ${\mathbf v}_0 = \overline{\tilde{\mathbf z}}_{N}$, then it follows with $\tilde{c}$ from \cref{Frobenius} that 
\begin{equation*}
\tilde{c} \, \tilde{\mathbf z}_{M} \tilde{\mathbf z}_{N}^{T} =
\tilde{\mathbf z}_{M}^{*} {\mathbf A} \overline{\tilde{\mathbf z}}_{N} \, \tilde{\mathbf z}_{M} \tilde{\mathbf z}_{N}^{T} 
= ({\mathbf u}_0^{*} {\mathbf A} {\mathbf v}_0) \, {\mathbf u}_0 \, {\mathbf v}_0^{*} = \sigma_{0} \, {\mathbf u}_0 \, {\mathbf v}_0^{*},
\end{equation*}
i.e., the unstructured and the structured rank-1 approximation coincide.

2.\ Assume that the structured low-rank approximation $\tilde{c} \, {\tilde{\mathbf z}}_{M} \, {\tilde{\mathbf z}}_{N}^{T}$ provides the optimal error in \cref{lowapp}. According to equation \cref{frob1} from the proof of \cref{theo1} we have
\begin{equation*}
    \left\| {\mathbf A} - \tilde{c} \, \tilde{\mathbf z}_{M} \tilde{\mathbf z}_{N}^{T} \right\|_{F}^{2} = \|{\mathbf A} \|_{F}^{2} - |\tilde{\mathbf z}_{M}^{*} \, {\mathbf A} \, \overline{\tilde{\mathbf z}}_{N}|^{2}
    \end{equation*}
and it follows on the one hand
\begin{equation*}
    \sigma_{0}^{2} = \|{\mathbf A} \|_{2}^{2} = |\tilde{\mathbf z}_{M}^{*} \, {\mathbf A} \, \overline{\tilde{\mathbf z}}_{N}|^{2}.
\end{equation*}
On the other hand,  the Theorem of Rayleigh-Ritz (see \cite{Horn85}, p.\ 176) implies 
\begin{equation*}
  | \tilde{\mathbf z}_{M}^{*} \, {\mathbf A} \, \overline{\tilde{\mathbf z}}_{N}|^{2}    %
    \stackrel{(\star)}{\le} \|\overline{\tilde{\mathbf z}}_{N} \, \tilde{\mathbf z}_{N}^{T} \|_{2} 
   \, \|{\mathbf A}^{*}  \tilde{\mathbf z}_{M} \|_{2}^{2}
    = \| {\mathbf A}^{*}  \tilde{\mathbf z}_{M} \|_{2}^{2}
    \stackrel{(\diamond)}{\le} 
    \|{\mathbf A}\|_{2}^{2}.
\end{equation*}
Here, equality at $(\star)$ only holds if ${\mathbf A}^{*} \tilde{\mathbf z}_{M}$ is an eigenvector of $\overline{\tilde{\mathbf z}}_{N} \, \tilde{\mathbf z}_{N}^{T}$ to the non-zero eigenvalue $\|\tilde{\mathbf z}_{N} \|_{2}^{2}$. Equality at $(\diamond)$ is achieved if moreover $\tilde{\mathbf z}_{M}$ is an eigenvector of ${\mathbf A} {\mathbf A}^{*}$  to the largest eigenvalue $\sigma_{0}^{2}= \|{\mathbf A} \|_{2}^{2}$. The assertion now follows by comparison with \cref{sing}.
\end{proof}

In the remainder of this section, we will derive  further  properties of the optimal value $\tilde{z}$ in \cref{Frobenius} in order to provide an efficient algorithm to compute $\tilde{z}$ and $\tilde{c}$. 
First we consider the possible range of $\tilde{z}$. 
For this purpose, we recall that a rank-1 Hankel matrix $\mathbf{H}_{1}$ can also  be represented as 
$\mathbf{H}_{1} = c \, {\mathbf w}_{M} \, {\mathbf w}_{N}^{T}$ or $\mathbf{H}_{1} = c \, \tilde{\mathbf e}_{M} \tilde{\mathbf e}_{N}^{T}$
with ${\mathbf w}_{N}= {\mathbf w}_{N}(z)$ in \cref{w}, and $\tilde{\mathbf e}_{N}$ in (\ref{en}).

\begin{theorem}\label{theoneu}
Let ${\mathbf A} \in {\mathbb C}^{M \times N}$ with $M,N \ge 2$ and $\rank({\mathbf A}) \ge 1$. 
Define
\begin{equation*}
    F(z) \coloneqq {\mathbf z}_{M}^{*} \, {\mathbf A} \overline{\mathbf z}_{N}, \qquad 
    F_{1}(z) \coloneqq {\mathbf z}_{M}^{*} \, {\mathbf J}_{M}\, {\mathbf A} \, {\mathbf J}_{N}\, \overline{\mathbf z}_{N}\end{equation*}
for $z \in {\mathbb C}$ with ${\mathbf J}_{N}$ as in \cref{J}. 
Let $M_{0}\coloneqq \max_{|z|\le 1} |F(z)|$ and $M_{1}\coloneqq \max_{|z|\le 1} |F_{1}(z)|$.
Then the optimal rank-1 Hankel approximation $\mathbf{H}_{1} = \tilde{c} \, \tilde{\mathbf z}_{M} \, \tilde{\mathbf z}_{N}^{T}$ of ${\mathbf A}$ is determined by
\begin{equation*}
    \tilde{z} \, \in
    \begin{cases}
        \mathop{\mathrm{argmax}}_{|z|\leq 1} |F(z)| & \text{if } M_0 \geq M_1,\\[1ex]
        \left(\mathop{\mathrm{argmax}}_{|z|<1} |F_1(z)| \right)^{-1} & \text{if } M_1 > M_0,
    \end{cases}
    \qquad
    \tilde{c} \coloneqq  \tilde{\mathbf z}_{M}^{*} {\mathbf A} \overline{\tilde{\mathbf z}}_{N}.
\end{equation*}
\end{theorem}

\begin{proof}
We show that $|F_{1}(z)| = |F(1/z)|$, then the assertion of the theorem follows from \cref{theo1}.
We observe from (\ref{w}) that
\begin{equation*}
    {\mathbf J}_{N} {\mathbf z}_{N}(z)  
            = {\mathbf z}_N\left(1/z\right) = {\mathbf w}_{N}(z) \quad \textrm{for} \; z \neq 0 \quad \textrm{and} \quad  {\mathbf J}_{N} {\mathbf e}_{N}= \tilde{\mathbf e}_{N}.
\end{equation*}
Thus, we conclude that  
\begin{equation*}
    |F_{1}(z)|
    =|(\mathbf J_M\mathbf z_{M}(z))^{*}  {\mathbf A}   (\mathbf J_N \overline{\mathbf z}_{N}(z))|    %
    =  |({\mathbf z}_{M}(1/z))^{*}  {\mathbf A} \,  \overline{\mathbf z}_{N}(1/z)|
    = \left|F\left(1/z \right)\right|
\end{equation*}
for $z \neq 0$ as well as
\begin{equation*}
    |F_{1}(0)| = |({\mathbf J}_M {\mathbf e}_{M})^{*}  {\mathbf A}   ({\mathbf J}_N \overline{\mathbf e}_{N})| = 
    |(\tilde{\mathbf e}_{M})^{*}  {\mathbf A}   \overline{\tilde{\mathbf e}}_{N}| = \lim_{z \to 0} \left|F\left(1/z\right)\right|.
\end{equation*}
The assertion now follows from \cref{theo1}.
\end{proof}

\begin{remark}
Using \cref{theoneu},  we can restrict the search for an optimal value $\tilde{z}$ to the unit disc $\{ z: |z| \le 1 \}$ if we consider the two functions $F(z)$ and $F_{1}(z)$.
\end{remark}

\subsection{Real Rank-1 Hankel Approximations}

In the following we will consider real matrices ${\mathbf A} \in {\mathbb R}^{M \times N}$ and restrict the search to real optimal rank-1 Hankel approximations, i.e., we search for real parameters $\tilde{c}$ and $\tilde{z}$. Then we can derive further conditions on $\tilde{z}$ that simplify the computation of the optimal rank-1 Hankel approximation of ${\mathbf A}$. A similar approach in a weighted Frobenius norm has been presented in \cite{deMoor94}.

\begin{theorem} \label{theo3}   
Let ${\mathbf A} \in {\mathbb R}^{M \times N}$  with $M,N \ge 2$, $\rank({\mathbf A}) \ge 1$, and  $|{a_{0,0}}| \ge |{a_{M-1,N-1}}|$. If $H_{1} = \tilde{c} \, \tilde{\mathbf z}_{M} \tilde{\mathbf z}_{N}^{T}$ is an optimal rank-1 Hankel approximation of ${\mathbf A}$, then 
\begin{equation*}
    Q(\tilde z) \coloneqq a'(\tilde z)\, p(\tilde z) - a(\tilde z) \, p'(\tilde z) =0,
\end{equation*}
with  
\begin{align*}
    a( z) \coloneqq \sum_{j=0}^{N-1} \sum_{k=0}^{N-1} a_{j,k} \, z^{j+k}, \qquad 
    p(z) \coloneqq \left(\sum_{k=0}^{M-1} z^{2k} \right)^{1/2} \left(\sum_{k=0}^{N-1} z^{2k}\right)^{1/2} \ge 1.
\end{align*}
Here, $a'(z)$ and $p'(z)$ denote the first derivatives of $a(z)$ and $p(z)$, respectively.
\end{theorem}

\begin{proof}  
According to \cref{theo1} we obtain $\tilde{z}$ as 
\begin{equation*}
    \tilde{z} \in  \mathop{\mathrm{argmax}}_{z \in {\mathbb R}} |F(z)|
    \qquad \textrm{with} \qquad 
    F(z) = {\mathbf z}_{M}^{T} {\mathbf A} {\mathbf z}_{N} = \frac{a(z)}{p(z)}.
\end{equation*}
Thus, $F(\tilde{z})$ is an extremal value of $F$, i.e., $F'(\tilde{z}) =0$.
The first derivative of $F$ is given by
\begin{equation*}
    F'(z) =  \left( \frac{a'(z)p(z)- a(z)p'(z)}{p(z)^{2}} \right).
\end{equation*}
Since $p(z)\ge 1$ for all $z \in {\mathbb R}$, we obtain for $\tilde{z}$ the necessary condition 
\begin{equation*}
    a'(\tilde z)p(\tilde z)- a(\tilde z)p'(\tilde z)= 0,
\end{equation*}
as was claimed.
\end{proof}

Considering the monomial representation of the polynomial
\begin{equation}\label{pola}
    a(z) = \sum_{j=0}^{M-1} \sum_{k=0}^{N-1} a_{j,k} \, z^{j+k}
    \eqqcolon \sum_{\ell=0}^{M+N-2} h_{\ell} \, z^{\ell}
\end{equation}
with $h_{\ell} := \sum_{j+k=\ell} a_{j,k}$ as in \cref{rem3.2}(3), we can conclude even more.

\begin{corollary}\label{coro1}
Let ${\mathbf A}$ be a real $M \times N$ matrix 
with $M,N \ge 2$, $\rank({\mathbf A}) \ge 1$, and $|{a_{0,0}}| \ge |{a_{M-1,N-1}}|$ .
Let $a(z)$ be given as in \cref{pola} and 
\begin{equation*}
    \tilde{z} \in \mathop{\mathrm{argmax}}_{z \in {\mathbb R}}  ({\mathbf z}_{M}^{T} {\mathbf A} {\mathbf z}_{N})^{2}= \mathop{\mathrm{argmax}}_{z \in {\mathbb R}} \left( \frac{a(z)}{p(z)} \right)^{2}.
\end{equation*}
\begin{description}
  \item{(1)} If $h_{\ell} \ge 0$ for $\ell=0, \ldots , M+N-2$ and $h_{0} \ge h_{M+N-2}$, then there exists $\tilde{z} \ge 0$.
  \item{(2)} If $h_{\ell} \ge 0$ for $\ell$ even and $h_{\ell}\le 0$ for $\ell$ odd, then there exists $\tilde{z} \le 0$. 
\end{description}
\end{corollary}

\begin{proof} The first assertion follows directly from the observation that $a(z) \ge a(-z)$ for $z  \ge 0$ while $p(z) = p(-z)$ is an even function. In the second case we have $a(-z) \ge a(z)$ for all $z\ge 0$, and the assertion follows similarly.
\end{proof}

\begin{remark}
If we assume additionally that ${\mathbf A}$ is square, i.e., $M = N$, and the coefficients of the corresponding polynomial $a(z)$ of ${\mathbf A}$ in \cref{pola} 
are non-negative and monotonically decreasing, then 
it can be shown that there exists $\tilde{z}$ in $(0,1)$
that generates the optimal rank-1 Hankel matrix, and this value $\tilde{z}$ is the only positive zero of $Q(z)$ in \cref{theo3}. In this case we can find $\tilde{z}$ efficiently by employing a Newton method with starting value $z_{0}= 1$. 
\end{remark}

Finally, we want to answer the following question: Given a real matrix ${\mathbf A}$, can we restrict the search for the optimal parameters $c$ and $z$ to real numbers, or can we achieve better results by allowing complex parameters? The following example shows that indeed complex parameters may provide better approximations.

\begin{example} 
We want to find an optimal  rank-1 Hankel approximation for
\begin{equation*}
{\mathbf A} =
  \begin{pmatrix}
    1 & -\frac{1}{2} & -1 \\
    -\frac{1}{2} & -1  &-\frac{1}{2} \\
    -1 & -\frac{1}{2} & 1
  \end{pmatrix}.
\end{equation*}
This matrix has the eigenvalues $2.0$, $-1.366025$ and $0.366025$ and the Frobenius norm $\|{\mathbf A}\|_{F} = 2.449490$.
Using \cref{theo3} we find two solutions for the optimal real parameters, $(\tilde{z}, \, \tilde{c})=  (-0.129135, \, 1.063508)$ and  $(\tilde{z}, \, \tilde{c})= (-7.743849, 1.063508)$. The obtained Frobenius norm of the error is $\|{\mathbf A} - \tilde{c} \, \tilde{\mathbf z} \, \tilde{\mathbf z}^{T}\|_{F} = 2.206570$ (for both solutions).

If we allow $\tilde{c}$ and $\tilde{z}$ to be complex, we obtain with $(\tilde{z}, \, \tilde{c})  =({\mathrm i}, \, \frac{5}{3})$ as well as with $(\tilde{z}, \, \tilde{c})  =(-{\mathrm i}, \, \frac{5}{3})$ the smaller error  $\|{\mathbf A} - \tilde{c} \, \tilde{\mathbf z} \, \tilde{\mathbf z}^{T}\|_{F} = \frac{\sqrt{261}}{9} =1.7950055$.
\end{example}

\begin{remark}
For non-negative matrices ${\mathbf A}$ with $a_{0,0} \ge a_{M-1,N-1}$, there always exists an optimal rank-1 Hankel approximation with real non-negative parameters $\tilde{z}$ and $\tilde{c}$. In this case, the polynomial $a(z)$ in (\ref{pola}) has only non-negative coefficients $h_{\ell}$ and we obtain
\begin{equation*}
    |{\mathbf z}_{M}^{T} {\mathbf A} {\mathbf z}_{N}|=  \frac{\left|\sum\limits_{\ell=0}^{M+N-2} h_{\ell} z^{\ell}\right|}{p(z)} \le \frac{\sum\limits_{\ell=0}^{M+N-2} h_{\ell} |z|^{\ell}}{p(z)}
\end{equation*}
with $p(z)$ in \cref{theo3}.
This term can therefore be maximized by a real value $\tilde{z} \ge 0$. Then $\tilde{c} ={\mathbf z}_{M}^{T} {\mathbf A} {\mathbf z}_{N}$ is also real and non-negative.
\end{remark}

\section{Optimal Rank-1 Hankel Approximation in the Spectral Norm}
\label{sec:spec}

We consider now the minimization problem \cref{problem1} in the spectral norm, which is much more difficult to solve than  problem \cref{problem} for the Frobenius norm. 
Therefore, we restrict our considerations to real symmetric matrices ${\mathbf A} \in {\mathbb R}^{N \times N}$ 
and show how to obtain the real optimal rank-1 Hankel approximation $\mathbf{H}_{1}$ in this case.
\Cref{thm:rank-1_hankel}  implies that for a symmetric matrix ${\mathbf A} \in {\mathbb R}^{N \times N}$ we need to solve the minimization problem  
\begin{equation}\label{probs}
	\min_{c,z\in \overline{\mathbb R}} \left\|{\mathbf A}- c\, {\mathbf z} \, {\mathbf z}^{T}\right\|^2_2
\end{equation}
with $\overline{\mathbb R} = {\mathbb R} \cup \{\infty\}$ and ${\mathbf z} = {\mathbf z}_{N}$ as in \cref{z} with $\lim_{z \to \infty} {\mathbf z}_N = {\mathbf e}_{N}$.

The real symmetric matrix $\mathbf A$ possesses an eigendecomposition ${\mathbf A} = {\mathbf V} \, {\mathbf \Lambda} {\mathbf V}^{T}$ with an orthogonal matrix
${\mathbf V}$ whose columns ${\mathbf v}_{j}$ are the eigenvectors of ${\mathbf A}$, and with the diagonal matrix ${\mathbf  \Lambda}= \diag(\lambda_{0}, \ldots , \lambda_{N-1})$, where the real eigenvalues are ordered by modulus $|\lambda_{0}| \ge |\lambda_{1}| \ge |\lambda_{2}| \ge \ldots \ge |\lambda_{N-1}|$. Note that ${\mathbf A} \, {\mathbf v}_{j} = \lambda_{j} \, {\mathbf v}_{j}$ for $j=0, \ldots , N-1$, and $\| {\mathbf A}\|_{2} = |\lambda_{0}|$. Without loss of generality, we assume that $\lambda_{0} = |\lambda_{0}| >0$.

 Our goal is to find  necessary and sufficient conditions for optimal parameters $\tilde{z}$ and $\tilde{c}$ such that 
$\tilde{c} \, \tilde{\mathbf z} \, \tilde{\mathbf z}^{T}$ solves the minimization problem \cref{probs}. 
Let  $\tilde{c} \, \tilde{\mathbf z} \, \tilde{\mathbf z}^{T}$ denote an optimal solution of \cref{probs}, and let 
\begin{equation}\label{lati}
    \tilde{\lambda} \coloneqq \left\|{\mathbf A} - \tilde{c}\,  \tilde{\mathbf z} \, \tilde{\mathbf z}^{T} \right\|_{2}
\end{equation}
denote the corresponding optimal approximation error.
Thus the two matrices 
\begin{equation*}
    \tilde{\lambda}\, {\mathbf I}_{N}-{\mathbf A} + \tilde{c}\,  \tilde{\mathbf z} \, \tilde{\mathbf z}^{T} \qquad \textrm{ and } \qquad 
    \tilde{\lambda}\, {\mathbf I}_{N}- \tilde{c}\,  \tilde{\mathbf z} \, \tilde{\mathbf z}^{T} + {\mathbf A}
\end{equation*}
are positive semidefinite, and at least one of these two matrices is singular.
Let $\mubf \coloneqq {\mathbf V}^{T} \, \tilde{\mathbf z} = (\mu_{0}, \ldots , \mu_{N-1})^{T}$, i.e., $\mu_{j} = {\mathbf v}_{j}^{T} \tilde{\mathbf z}$, for $ j=0, \ldots , N-1$.
Then \cref{lati} implies that 
\begin{align}\label{m1}
    {\mathbf M}_{1}(\tilde{\lambda}) &\coloneqq {\mathbf V}^{T}(\tilde{\lambda}\, {\mathbf I}_{N}- {\mathbf A} + \tilde{c}\,  \tilde{\mathbf z} \, \tilde{\mathbf z}^{T}) {\mathbf V} = 
    \tilde{\lambda} \, {\mathbf I}_N - {\mathbf \Lambda} + \tilde{c} \,  \mubf \, \mubf^{T}, \\
    \label{m2}
    {\mathbf M}_{2}(\tilde{\lambda}) & \coloneqq {\mathbf V}^{T}(\tilde{\lambda}\, {\mathbf I}_{N}- \tilde{c}\,  \tilde{\mathbf z} \, \tilde{\mathbf z}^{T} + {\mathbf A}) {\mathbf V}=
    \tilde{\lambda} \, {\mathbf I}_N + {\mathbf \Lambda} - \tilde{c} \,  \mubf \, \mubf^{T} 
\end{align}
are positive semidefinite, and at least one of them possesses the eigenvalue $0$. 
Note that ${\mathbf M}_{1}(\tilde{\lambda})$ and ${\mathbf M}_{2}(\tilde{\lambda})$ have a special structure, namely a sum of a diagonal matrix and a rank-1 matrix. Therefore, we first investigate conditions for the definiteness of such matrices in \cref{defi}.  These observations will enable us to prove \cref{theospectral1}, which provides necessary and sufficient conditions for the optimal parameters $\tilde{z}$ and $\tilde{c}$ solving \cref{probs} in \cref{sec:app}.
\cref{theospectral1} in turn gives rise to an algorithm for computing the optimal rank-1 Hankel approximation with regard to the spectral norm, which is presented in \cref{last}.

\subsection{Definiteness of Diagonal-Plus-Rank-1-Matrices}
\label{defi}
The two matrices ${\mathbf M}_{1}(\tilde{\lambda})$ and ${\mathbf M}_{2}(\tilde{\lambda})$ in (\ref{m1}) and (\ref{m2}) are both of the form 
\begin{equation*}
    \mathbf B = \mathbf D + c\mathbf{bb}^T,
\end{equation*}
where $\mathbf D = \diag(d_0,\dots,d_{N-1})\in\mathbb R^{N\times N}$, $\mathbf b=(b_{0}, \ldots , b_{N-1})^{T}\in \mathbb R^N$, and $c\in\mathbb R$.
In view of ${\mathbf D} = \tilde{\lambda} {\mathbf I}_{N} + {\mathbf \Lambda}$  and ${\mathbf D} = \tilde{\lambda} {\mathbf I}_{N} - {\mathbf \Lambda}$ for ${\mathbf M}_{1}(\tilde{\lambda})$ and ${\mathbf M}_{2}(\tilde{\lambda})$, we 
 are especially interested in the two cases where either the fixed diagonal matrix ${\mathbf D}$ is positive semidefinite or  ${\mathbf D}$ possesses exactly one negative diagonal entry while all other diagonal entries are non-negative.
For these two cases, we derive necessary and sufficient conditions for ${\mathbf b}$ and $c$, such that $\mathbf B$ is positive semidefinite.
We start with an observation for the determinant of ${\mathbf B}$.

\begin{lemma}\label{lem1}
The matrix ${\mathbf B} \coloneqq {\mathbf D} + c \, {\mathbf b} {\mathbf b}^{T}$ has the determinant
\begin{equation*}
    \det ({\mathbf B}) =  \det({\mathbf D}) + c  \, \sum_{j=0}^{N-1} b_{j}^{2} \, \left(\prod_{\substack{k=0\\ k\neq j}}^{N-1} d_{k}\right).
\end{equation*}
If ${\mathbf D}$ is invertible, we have
\begin{equation*}
    \det ({\mathbf B}) = \det ({\mathbf D}) \left( 1+ c  \sum_{j=0}^{N-1} \frac{b_{j}^{2}}{d_{j}}\right).
\end{equation*}
\end{lemma}

\begin{proof}
We employ the rule for computing determinants of block matrices, see \cite{Silvester00},
\begin{equation*}
    \det \begin{pmatrix}
             \mathbf{D} & -\mathbf{b} \\
             c \mathbf{b}^T & 1
         \end{pmatrix}
    = \det (1 \cdot {\mathbf D}+ c \, {\mathbf b} \, {\mathbf b}^{T}) = \det({\mathbf B}) ,
\end{equation*}
and use an expansion of the determinant with respect to the last column. For the case where $\mathbf D$ is invertible, see also \cite{Demmel1997}.
\end{proof} 

For the remainder of Section \ref{sec:spec}, we use the convention that in the case $b_{j}= d_{j}=0$ the term $\frac{b_{j}^{2}}{d_{j}}$ in the sum $\sum\limits_{k=0}^{N-1} \frac{b_{k}^{2}}{d_{k}}$ is just omitted, and to remind the reader that such terms may occur in the sum, we will use the notation $\sumprime_{}^{}$ instead of $\sum$.

\begin{lemma}\label{lem2}
    Let $N\geq2$,  $c>0$, and assume that $\mathbf D$ has only one negative eigenvalue, i.e., $d_0 < 0$ and $d_j \geq 0$ for $j=1, \ldots, N-1$. Then the matrix $\mathbf B = \mathbf D + c\mathbf{bb}^T$ is positive semidefinite if and only if $\mathbf b$ and $c$ satisfy
    \begin{equation}\label{lem21}
        \sumprime_{j=0}^{N-1} \frac{b_j^{2}}{(-d_j)} \ge \frac{1}{c},
    \end{equation}
    where $b_j=0$ whenever $d_j=0$.
    Moreover, if $d_j>0$ for $j\geq1$ and the inequality \cref{lem21} is strict, then $\mathbf B$ is  positive definite.
\end{lemma}

\begin{proof}
According to \cite{Prussing86}, the matrix $\mathbf B$ is positive semidefinite if and only if all its principal minors, i.e., the determinants of all possible $r\times r$ principal submatrices of ${\mathbf B}$ for $r=1, \ldots , N$, are non-negative. We observe that all these principal submatrices of $\mathbf B$ are of the same form as $\mathbf B$. 
For  all $j \in \{1, \ldots , N-1\}$ with $d_{j}=0$, we consider the $(2\times2)$ submatrices $\mathbf B_{\{0,j\}}$ obtained from the first and $j$-th row and column of $\mathbf B$. The determinant of $\mathbf B_{\{0,j\}}$ is given by
    \begin{equation*}
        \det (\mathbf B_{\{0,j\}}) = \textrm{det}
        \begin{pmatrix}
            d_{0}+ c \, b_{0}^{2} & c\, b_{0} \, b_{j} \\ c\, b_{0} \, b_{j} & c \, b_{j}^{2}
        \end{pmatrix}
        = cb_j^2 d_0,
    \end{equation*}
    which is negative as long as $b_j\neq0$. Hence, the condition $ \det(\mathbf B_{\{0,j\}}) \ge 0$ implies that $b_j=0$ for all indices $j$ with $d_j=0$.
    Now let $J$ be the index set containing all indices corresponding to non-zero entries of $\mathbf D$. Observe that by assumption $0\in J$. Denote by $\mathbf B_J$ and $\mathbf D_J$ the corresponding principal submatrices of $\mathbf B$ and $\mathbf D$, respectively. By \cref{lem1}, we have
    \begin{equation*}
        \det ({\mathbf B_J}) = \det ({\mathbf D_J}) \left( 1+ c  \sum_{j\in J} \frac{b_{j}^{2}}{d_{j}}\right),
    \end{equation*}
    where $\det(\mathbf D_J) < 0$. Thus, the condition $\det (\mathbf B_J) \ge 0$ is equivalent to
    \begin{equation*}
        \left( 1+ c  \sum_{j\in J} \frac{b_{j}^{2}}{d_{j}}\right) \leq 0 
        \quad \Leftrightarrow \quad
        \sum_{j\in J} \frac{b_j^2}{(-d_j)} \geq \frac1c.
    \end{equation*}
  These conditions are already sufficient for all principal minors of $\mathbf B$ corresponding to subsets of $J$ to be non-negative.
    By adding the zero terms corresponding to indices not in $J$ to the above inequality, the first claim follows. \\
    If $d_j>0$ for $j\geq1$ and the inequality (\ref{lem21}) is strict, then already all leading principal minors of $\mathbf B$ are positive, which is equivalent to $\mathbf B$ being positive definite.
\end{proof}
\begin{lemma}\label{lem3}
    Let $N\geq2$, $c>0$, and assume that  $\mathbf D$ is positive semidefinite with at least one positive eigenvalue, i.e., $d_0>0$ and $d_j\geq0$ for $j=1, \ldots, N-1$. Then the matrix $\mathbf B = \mathbf D - c\mathbf{bb}^T$ is positive semidefinite if and only if $\mathbf b$ and $c$ satisfy
    \begin{equation}\label{lem31}
        \sumprime_{j=0}^{N-1} \frac{b_j^2}{d_j} \leq \frac{1}{c},
    \end{equation}
    where $b_j=0$ whenever $d_j=0$.
    Moreover, if $d_j>0$ for all $j$ and the inequality \cref{lem31} is strict, then $\mathbf B$ is positive definite.
\end{lemma}
\begin{proof}
As in the proof of \cref{lem2} we study the principal minors of $\mathbf B$. For all indices $j$ with $d_{j}=0$ we consider the $(2\times2)$ submatrices $\mathbf B_{\{0,j\}}$ obtained by the first and $j$-th row and column of $\mathbf B$. We find
    \begin{equation*}
        \det (\mathbf B_{\{0,j\}}) = \textrm{det}
        \begin{pmatrix}
            d_{0}- c \, b_{0}^{2} & -c\, b_{0} \, b_{j} \\ -c\, b_{0} \, b_{j} & -c \, b_{j}^{2}
        \end{pmatrix}
        = -cb_j^2 d_0,
    \end{equation*}
 and $\det (\mathbf B_{\{0,j\}}) \ge 0$ implies that $b_j=0$.
    Let again $J$ be the index set containing all indices corresponding to non-zero entries of $\mathbf D$, and denote by $\mathbf B_J$ and $\mathbf D_J$ the corresponding principal submatrices of $\mathbf B$ and $\mathbf D$, respectively. By \cref{lem1}, we have
    \begin{equation*}
        \det (\mathbf B_J) = \det (\mathbf D_J) \left( 1- c  \sum_{j\in J} \frac{b_{j}^{2}}{d_{j}}\right),
    \end{equation*}
    where $\det (\mathbf D_J) > 0$. Thus, $\det (\mathbf B_J) \ge 0$  is equivalent to
    \begin{equation*}
        \left( 1- c  \sum_{j\in J} \frac{b_{j}^{2}}{d_{j}}\right) \geq 0 
        \quad \Leftrightarrow \quad
        \sum_{j\in J} \frac{b_j^2}{d_j} \leq \frac1c.
    \end{equation*}
    These conditions are already sufficient for all principal minors of $\mathbf B$ to be non-negative.
    By adding the zero terms corresponding to indices not in $J$ to the above inequality, the first claim follows. 
If $d_j>0$ for all $j$ and the inequality (\ref{lem31}) is strict, then all leading principal minors are positive and thus $\mathbf B$ is positive definite.
\end{proof}

\subsection{The Optimal Approximation Error in the Spectral Norm}
\label{sec:app}
First, we present upper and lower bounds for the optimal rank-1 Hankel approximation error $\tilde{\lambda}$ in (\ref{lati}), where we assume $\lambda_{0} = \|{\mathbf A}\|_{2} > |\lambda_{1}|$, i.e. $\lambda_{0}$ is a single singular value of ${\mathbf A}$. The case $\lambda_{0} =|\lambda_{1}|$ will be treated separately.
\begin{proposition}\label{bounds}
    Let $\tilde\lambda = \left\|{\mathbf A} - \tilde{c}\,  \tilde{\mathbf z} \, \tilde{\mathbf z}^{T} \right\|_{2}$ be the optimal rank-1 Hankel approximation error in $(\ref{lati})$. Then we have
    \begin{equation*}
        \abs{\lambda_1} \leq \tilde\lambda < \lambda_0,
    \end{equation*}
    where $\lambda_0$ and $\lambda_1$ are the largest and (by modulus) second largest eigenvalue of $\mathbf A$.
\end{proposition}

\begin{proof}
By the Eckart-Young-Mirsky theorem it follows immediately that $\tilde{\lambda} \ge |\lambda_{1}|$.
We prove that the optimal error $\tilde{\lambda}$ satisfies $\tilde{\lambda} < \lambda_{0}$. Assume by contrast that $\tilde{\lambda} = \lambda_{0}$.
We show that there exist parameters $\tilde{z}$ and $\tilde{c}$ such that ${\mathbf M}_{1}(\tilde{\lambda})$ and ${\mathbf M}_{2}(\tilde{\lambda})$ in (\ref{m1}) and (\ref{m2}) are both positive definite, which leads to a contradiction, since for the optimal parameters $\tilde{z}$, $\tilde{c}$, at least one of these two matrices has to be singular.
Choose for example some $\tilde{z}\in {\mathbb R}$ such that $\mu_{0} = {\mathbf v}_{0}^{T} \tilde{\mathbf z} \neq 0$, and choose $\tilde{c} = \frac{\lambda_{0} - |\lambda_{1}|}{2} 
$.
Then  ${\mathbf M}_{1}(\tilde{\lambda})= (\lambda_{0} {\mathbf I}_{N} - {\mathbf \Lambda}) + \tilde{c} \, \mubf \, \mubf^{T}$ is a sum of two positive semidefinite matrices, and Lemma \ref{lem1} yields
\begin{equation*}
    \det {\mathbf M}_{1}(\tilde{\lambda}) = \tilde{c} \sum_{j=0}^{N-1} \mu_{j}^{2} \, \Bigg(\prod_{\substack{k=0\\ k\neq j}}^{N-1} (\lambda_{0} - \lambda_{k})\Bigg) >0.
\end{equation*}
Further, ${\mathbf M}_{2}(\tilde{\lambda})= (\lambda_{0} {\mathbf I}_{N} + {\mathbf \Lambda}) - \tilde{c} \, \mubf \, \mubf^{T}$ is a 
difference of a positive definite matrix $(\lambda_{0} {\mathbf I}_{N} + {\mathbf \Lambda})$ 
with smallest possible eigenvalue $\lambda_{0} - |\lambda_{1}|>0$ and a positive semidefinite
rank-1 matrix $\tilde{c} \, \mubf \, \mubf^{T}$ with nonzero 
eigenvalue $\tilde{c} \,  \| \mubf\|_{2}^{2} = \tilde{c} = \frac{\lambda_{0} - |
\lambda_{1}|}{2}$. Thus, Weyl's interlacing inequality, see Corollary 4.3.3 in 
\cite{Horn85},  implies that ${\mathbf M}_{1}(\tilde{\lambda})$ is positive 
definite, since we again observe that $\det {\mathbf M}_{1}(\tilde{\lambda}) >0$ using Lemma \ref{lem1}.
\end{proof}
In order to state our main theorem, we introduce the following function. Recalling that $\lim\limits _{z \to \infty}  {\mathbf z} = {\mathbf e}_{N}$  let for $z \in \overline{\mathbb R} \coloneqq {\mathbb R} \cup \{\infty\}$ and $\lambda^2 \in [\lambda_1^2, \lambda_0^2)$,
\begin{equation}\label{f}
    f(z, \lambda^{2}) \coloneqq  \sumprime_{j=0}^{N-1}
    \frac{({\mathbf v}_{j}^{T} {\mathbf z})^{2}}{\lambda_{j}^{2} - \lambda^{2}} = \sumprime_{j=0}^{N-1}
    \frac{\mu_{j}^{2}}{\lambda_{j}^{2} - \lambda^{2}}
    ={\mathbf z}^{T} ({\mathbf A}^{2} - \lambda^{2} {\mathbf I})^{-1} {\mathbf z} .
\end{equation}
The last equality in \cref{f} follows with ${\mathbf z} = \sum\limits_{j=0}^{N-1} ({\mathbf v}_j^T {\mathbf z}) \, {\mathbf v}_j$ from the observations
\begin{equation*}
{\mathbf z}^{T} \left( {\mathbf A}^{2} - \lambda^{2} {\mathbf I}_{N} \right)^{-1} {\mathbf z} = \sum_{j=0}^{N-1} \left( {\mathbf v}_j^T {\mathbf z} \right)^2 \left( {\mathbf v}_j^T \left( {\mathbf A}^2 - \lambda^2 {\mathbf I} \right)^{-1} {\mathbf v}_j \right)
\end{equation*}
and $(\lambda_j^2- \lambda^2)^{-1} = {\mathbf v}_j^T ({\mathbf A}^2 - \lambda^2 {\mathbf I}_{N})^{-1} {\mathbf v}_j$.
For $\lambda^{2} \in (\lambda_{1}^{2}, \lambda_{0}^{2})$ and $z \in \overline{\mathbb R}$, $f(z, \lambda^{2})$ is well defined and bounded from below and above by the smallest and largest eigenvalue of $({\mathbf A}^{2} - \lambda^{2} \, {\mathbf I}_{N})^{-1}$, respectively. 
More exactly, we have
\begin{equation*}
    \min \left\{ (\lambda^{2}-\lambda_{N-1}^{2})^{-1}, (\lambda_{0}^{2} - \lambda^{2})^{-1} \right\}
    \le f(z, \lambda^{2}) \le
    \max \left\{ (\lambda_{0}^{2} - \lambda^{2})^{-1}, \, (\lambda^{2} - \lambda_{N-1}^{2})^{-1} \right\},
\end{equation*}
since it can be seen as a Rayleigh quotient for the matrix $({\mathbf A}^2 - \lambda^{2} {\mathbf I})^{-1}$.
Note, that for fixed $z\in\overline{\mathbb R}$, $f(z, \lambda^{2})$ is strictly monotonically increasing.
Moreover, for any fixed $z \in \overline{\mathbb R}$, 
$ f(z, \lambda_{0}^{2}) \coloneqq \lim_{\lambda^{2} \to \lambda_{0}^{2}} f(z, \lambda^{2})$ is bounded if and only if ${\mathbf v}_{0}^{T} {\mathbf z} =0$ (or ${\mathbf v}_{0}^{T} {\mathbf e}_{N} =0$ in case of $z=\infty$).
Similarly, $ f(z, \lambda_{1}^{2}) \coloneqq \lim_{\lambda^{2} \to \lambda_{1}^{2}} f(z, \lambda^{2})$ is bounded if and only if ${\mathbf v}_{j}^{T} {\mathbf z} =0$ (or ${\mathbf v}_{j}^{T} {\mathbf e}_{N} =0$ in case of $z=\infty$) for all ${\mathbf v}_{j}$ corresponding to eigenvalues $\lambda_{j}$ with $\lambda_{j}^{2} = \lambda_{1}^{2}$. In this case, we can extend the domain of $f$ to $[|\lambda_{1}|, \lambda_{0}]$ and use the convention $\frac00=0$ and the notation $\sumprime_{}^{}$ as before, if such terms terms occur in the sum.

The main theorem of this section contains two parts. The first part generalizes the result from \cite{Antoulas97} and states exact conditions ensuring that the rank-1 Hankel approximation achieves the same error as the unstructured rank-1 approximation. This error is given by $|\lambda_{1}|$ and achieved e.g.\ by truncated SVD. The second part of the theorem states necessary and sufficient conditions for the optimal parameters $ \tilde{z}$ and $\tilde{c}$ that enable us to derive an algorithm to compute the exact optimal rank-1 Hankel approximation if the error $|\lambda_{1}|$ cannot be achieved.

\begin{theorem}\label{theospectral1}
Let ${\mathbf A} =(a_{j,k})_{j,k=0}^{N-1,N-1}\in {\mathbb R}^{N \times N}$ be symmetric with $N \ge 2$.
Assume that $\rank({\mathbf A}) > 1$ and $\lambda_{0} = \|{\mathbf A}\|_{2} > |\lambda_{1}|$.
Let  the  optimal rank-1 Hankel approximation of ${\mathbf A}$ with regard to the spectral norm be of the form $\mathbf{H}_{1} = \tilde{c} \, \tilde{\mathbf z} \, \tilde{\mathbf z}^{T}$ with
 $   \tilde{\mathbf z} = {\Big(\sum\limits_{k=0}^{N-1} \tilde{z}^{2k}\Big)^{-1/2}}
    (1, \tilde{z}, \ldots , \tilde{z}^{N-1})^{T}$    
for $\tilde{z} \in {\mathbb R}$ or $\tilde{\mathbf z} = {\mathbf e}_{N}$ for $\tilde{z} = \infty$.

    (1) The optimal error bound $\| {\mathbf A} - {\mathbf H}_{1}\|_{2}^{2} = \| {\mathbf A} - \tilde{c} \, \tilde{\mathbf z} \, \tilde{\mathbf z}^{T}\|_{2}^{2} = \lambda_{1}^{2}$ is achieved if and only if there exists $\tilde{z} \in \overline{\mathbb R}$  such that 
    \begin{equation}\label{best} 
        {\mathbf v}_{j}^{T} \tilde{\mathbf z} =0, \text{ for all } \lambda_{j} \text{ with } |\lambda_{j}| = |\lambda_{1}|
        \quad \text{and}  \quad 
        f(\tilde{z},\lambda_{1}^{2}) \ge 0
            \end{equation}
   for $f(\tilde{z},\lambda_{1}^{2})$ in $(\ref{f})$, and if $\tilde{c}$ is chosen such that 
    \begin{equation}\label{cc} 
        \sumprime_{j=0}^{N-1} \frac{({\mathbf v}_{j}^{T} \tilde{\mathbf z})^{2}}{\lambda_{j} + |\lambda_{1}|}
        \le \frac{1}{\tilde{c}}
        \le \sumprime_{j=0}^{N-1} \frac{({\mathbf v}_{j}^{T} \tilde{\mathbf z})^{2}}{\lambda_{j} - |\lambda_{1}|}.
    \end{equation}

    (2) If there is no $\tilde{z}$ satisfying \cref{best}, then the optimal rank-1 Hankel approximation of ${\mathbf A}$ possesses the error 
    \begin{equation*}
        \tilde{\lambda} \coloneqq \| {\mathbf A} - {\mathbf H}_{1}\|_{2} = \left\| {\mathbf A} - \tilde{c} \, \tilde{\mathbf z} \, \tilde{\mathbf z}^{T} \right\|_{2} \in (|\lambda_{1}|, \, \lambda_{0}),
    \end{equation*}
    where   $\tilde{\lambda}$ is the minimal number in $(|\lambda_{1}|, \lambda_{0})$ satisfying  the relation
    \begin{equation} \label{spectral2}
        \max_{z \in \overline{\mathbb R}} f(z, \tilde{\lambda}^{2}) =0,
    \end{equation}
    and we have $\tilde{z} \in \mathop{\mathrm{argmax}}\limits_{z \in \overline{\mathbb R}} f(z, \tilde{\lambda}^{2})$. Further, 
    \begin{equation}\label{spectral3}
        \tilde{c} \coloneqq \left( \sum_{k =0}^{N-1} \frac{({\mathbf v}_{k}^{T} \tilde{\mathbf z})^{2}}{\lambda_{k} - \tilde{\lambda}} \right)^{-1}
        = \left( \tilde{\mathbf z}^{T} \left( {\mathbf A} - \tilde{\lambda} {\mathbf I} \right)^{-1} \tilde{\mathbf z} \right)^{-1}
        > 0.
    \end{equation}
\end{theorem}

\begin{proof}
Throughout this proof, let
\begin{equation}\label{la}
    \tilde{\lambda} \coloneqq \left\|{\mathbf A} - \tilde{c}\,  \tilde{\mathbf z} \, \tilde{\mathbf z}^{T} \right\|_{2}
\end{equation}
denote the optimal rank-1 approximation error, i.e., the parameters $\tilde{z}, \, \tilde{c}$ generate an optimal rank-1 Hankel approximation of ${\mathbf A}$. Recall from the beginning of this section that this holds if and only if the two symmetric matrices ${\mathbf M}_{1}(\tilde{\lambda})$ and ${\mathbf M}_{2}(\tilde{\lambda})$ in (\ref{m1}) and (\ref{m2})
are positive semidefinite, and at least one of them possesses the eigenvalue $0$.

1.\ 
The optimal parameter $\tilde{z}$ necessarily satisfies $\mu_{0}= {\mathbf v}_{0}^{T} \tilde{\mathbf z} \neq 0$, since otherwise we would find $({\mathbf A} - \tilde{c} \, \tilde{\mathbf z} \, \tilde{\mathbf z}^{T}) {\mathbf v}_{0} = {\mathbf A} {\mathbf v}_{0} = \lambda_{0} {\mathbf v}_{0}$ contradicting the upper bound $\tilde{\lambda} < \lambda_{0}$ from \cref{bounds}.
Moreover, we obtain the necessary condition  $\tilde{c} >0 $ since for $\tilde c\le0$ we would add a positive semidefinite matrix to ${\mathbf A}$ thereby enlarging the spectral norm,
\begin{align*}
    \|{\mathbf A} - \tilde{c} \, \tilde{\mathbf z} \, \tilde{\mathbf z}^{T}\|_{2}
    &= \max_{\| {\mathbf v}\|_{2} =1} |{\mathbf v}^{T}({\mathbf A} - \tilde{c} \, \tilde{\mathbf z} \, \tilde{\mathbf z}^{T}) {\mathbf v}| \\
    &\ge   {\mathbf v}_{0}^{T} {\mathbf A} {\mathbf v}_{0} - \tilde{c} \, ({\mathbf v}_{0}^{T} \tilde{\mathbf z})^{2} = \lambda_{0} + |\tilde{c}|\, ({\mathbf v}_{0}^{T} \tilde{\mathbf z})^{2} \ge \lambda_{0}.
\end{align*}

2.\ We derive necessary and sufficient conditions for the optimal parameters $\tilde{c}>0$, $\tilde{z} \in \overline{\mathbb R}$ and $\tilde{\lambda} \in [|\lambda_{1}|, \, \lambda_{0})$ by inspecting the matrices ${\mathbf M}_{1}(\tilde{\lambda})$ and ${\mathbf M}_{2}(\tilde{\lambda})$. Thereby we prove part (1) of \cref{theospectral1}.

Note that for the entries of the diagonal part $\tilde{\lambda} \, {\mathbf I} - {\mathbf \Lambda}$ of $\mathbf M_1(\tilde\lambda)$  in (\ref{m1}) we have $\tilde\lambda-\lambda_0 < 0$ while $\tilde\lambda-\lambda_1, \tilde\lambda-\lambda_2, \dots, \tilde\lambda-\lambda_{N-1} \geq 0$. The diagonal part 
$\tilde{\lambda} \, {\mathbf I} + {\mathbf \Lambda}$ of the matrix
$\mathbf M_2(\tilde\lambda)$ in (\ref{m2}) is positive semidefinite,  and we have $\tilde\lambda+\lambda_j \geq 0$ for all $j=0,\dots,N-1$ and actually $\tilde\lambda+\lambda_0>0$.
From \cref{lem2,lem3} it thus follows that ${\mathbf M}_{1}(\tilde{\lambda})$ and ${\mathbf M}_{2}(\tilde{\lambda})$ are both positive semidefinite if and only if 
\begin{equation}\label{c33}
    \sumprime_{j=0}^{N-1} \frac{\mu_{j}^{2}}{\lambda_{j} + \tilde{\lambda}}  \le \frac{1}{\tilde{c}}
    \le \sumprime_{j=0}^{N-1} \frac{\mu_{j}^{2}}{\lambda_{j} - \tilde{\lambda}}, 
\end{equation}
and if in case of $\tilde{\lambda} = |\lambda_{1}|$ moreover $\mu_{j} = {\mathbf v}_{j}^{T} \tilde{\mathbf z} =0$ for all $\lambda_{j}$ with $|\lambda_{j}| = \tilde{\lambda}$.
Obviously, a parameter $\tilde{c}$ satisfying \cref{c33} only exists, if 
\begin{equation*}
    \sumprime_{j=0}^{N-1} \frac{\mu_{j}^{2}}{\lambda_{j} - \tilde{\lambda}} - \sumprime_{j=0}^{N-1} \frac{\mu_{j}^{2}}{\lambda_{j} + \tilde{\lambda}}
    = 2 \tilde{\lambda} \sumprime_{j=0}^{N-1} \frac{\mu_{j}^{2}}{\lambda_{j}^{2} - \tilde{\lambda}^{2}} = 2 \tilde{\lambda} \, f(\tilde{z}, \tilde{\lambda}) \ge 0.
\end{equation*}
Observe that $\tilde{\lambda} \neq 0$ since $\tilde{\lambda} \ge |\lambda_{1}|$, and we have assumed that $\rank\mathbf A >1$. Thus, for the function $f(z,\lambda^{2})$ from \cref{f} it follows that $f(\tilde{z}, \tilde{\lambda}^{2}) \ge 0$, and we conclude \cref{best} and \cref{cc} for $\tilde{\lambda} = |\lambda_{1}|$.

3.\ We prove part (2) of \cref{theospectral1}. Assume that \cref{best} is not satisfied for any $\tilde{z} \in \overline{\mathbb R}$, i.e.,  
$\tilde{\lambda} > |\lambda_{1}|$. 
Inspecting the two sums in \cref{c33}, we observe that the left sum increases for decreasing $\tilde{\lambda}$ while the right sum decreases with  decreasing $\tilde{\lambda}$. Thus, \cref{c33} implies the equalities 
\begin{equation}\label{c4}
    \sum_{j=0}^{N-1} \frac{\mu_{j}^{2}}{\lambda_{j} + \tilde{\lambda}}  = \frac{1}{\tilde{c}}
    = \sum_{j=0}^{N-1} \frac{\mu_{j}^{2}}{\lambda_{j} - \tilde{\lambda}},
\end{equation}
for the minimal error $\tilde{\lambda}$. Otherwise, we could find a parameter $\tilde{c}$ such that the two inequalities in \cref{c33} are strict. But then, \cref{lem2,lem3} yield that the two matrices ${\mathbf M}_{1}(\tilde{\lambda})$ and ${\mathbf M}_{2}(\tilde{\lambda})$ are actually positive definite, and we could find some $|\lambda_1| \leq \lambda < \tilde{\lambda}$ such that ${\mathbf M}_{1}({\lambda})$ and ${\mathbf M}_{2}({\lambda})$ are still positive semidefinite. This would contradict our assumption \cref{la}.

Relation \cref{c4} directly implies that 
$\det {\mathbf M}_{1}(\tilde{\lambda}) = \det {\mathbf M}_{2}(\tilde{\lambda}) = 0$ by Lemma \ref{lem1}, or equivalently, that  $\tilde{\lambda}$ as well as $-\tilde{\lambda}$ are eigenvalues of ${\mathbf A} - \tilde{c} \, \tilde{\mathbf z} \, \tilde{\mathbf z}^{T}$.
Assertion \cref{spectral3} now follows from \cref{c4}.
Further, we conclude
\begin{equation*}
    \sum_{j=0}^{N-1} \frac{\mu_{j}^{2}}{\lambda_{j} + \tilde{\lambda}}  - \sum_{j=0}^{N-1} \frac{\mu_{j}^{2}}{\lambda_{j} - \tilde{\lambda}}
    = 2 \tilde{\lambda} \sum_{j=0}^{N-1} \frac{\mu_{j}^{2}}{\lambda_{j}^{2} - \tilde{\lambda}^{2}} = 0.
\end{equation*}
Since $\tilde{\lambda} > |\lambda_1| \geq 0$, this shows that $f(\tilde{z}, \tilde{\lambda}^{2}) = 0$.

Lastly, we consider $f(z, \tilde{\lambda}^{2})$ as a rational function in $z$ for the fixed  optimal error $\tilde{\lambda}$.  
We show that $f(z, \tilde{\lambda}^{2})\le 0$ for all $z \in \overline{\mathbb R}$.
Assume to the contrary that there is some $z$ with $f(z, \tilde{\lambda}^{2})> 0$. With the same arguments as before, we then obtain a range for the choice of $\tilde{c}$. But then $\tilde{c}$ can be taken such that the two matrices 
${\mathbf M}_{1}(\tilde{\lambda})$ and ${\mathbf M}_{2}(\tilde{\lambda})$ are positive definite. In that case $\tilde{\lambda}$ is no longer the optimal error, contradicting our assumption.
Thus we have shown \cref{spectral2}.
\end{proof}

\begin{remark}
1.\ The conditions \cref{best} in \cref{theospectral1} are particularly satisfied if the eigenvector corresponding to the largest eigenvalue $\mathbf v_0$ is of the form  ${\mathbf v}_{0} =\tilde{\mathbf z}$. In this case \cref{best} simplifies to
\begin{equation*}
    \frac{({\mathbf v}_{0}^{T} \tilde{\mathbf z})^{2}}{\lambda_{0}^{2} - \lambda_{1}^{2}} = \frac{\|\tilde{\mathbf z}\|_{2}^{2}}{\lambda_{0}^{2} - \lambda_{1}^{2}} = \frac{1}{\lambda_{0}^{2} - \lambda_{1}^{2}} \ge 0,
\end{equation*}
since ${\mathbf v}_{j}^{T} \tilde{\mathbf z} =0$ for $j=1, \ldots, N-1$.

2.\ The solution parameters $(\tilde{z}, \,\tilde{c})$ determining  the optimal rank-1 Hankel  approximation  with respect to the spectral norm need not be unique. If \cref{best} is satisfied and the optimal error $\Vert \mathbf A - \tilde{c}\tilde{\mathbf z}\tilde{\mathbf z}^T \Vert_2^2 = \lambda_1^2$ is attained, there are several possible choices for $\tilde c$ if the inequality in \cref{cc} is strict. If \cref{best} cannot be satisfied and $\tilde c$ is determined uniquely by \cref{spectral3}, it may happen that $\tilde{z} \in  \mathop{\mathrm{argmax}}_{z \in \mathbb{R}} f(z,\tilde \lambda^2)$ is not unique, see \cref{Cadex}.
\end{remark}

Finally, we study the problem \cref{probs} if the largest singular value $\lambda_{0}$ of ${\mathbf A}$ occurs with higher multiplicity.

\begin{corollary}
Let ${\mathbf A} = (a_{j,k})_{j,k=0}^{N-1,N-1} \in {\mathbb R}^{N \times N}$ be symmetric  with $N\ge 2$. Assume that $\textrm{rank} ({\mathbf A}) >1$ and let $\lambda_{0} = \| {\mathbf A}\|_{2} = |\lambda_{1}|$. \\
If all eigenvalues $\lambda$ of ${\mathbf A}$ with $|\lambda| = \|{\mathbf A}\|_{2}$  have the same sign, then every rank-1 Hankel matrix $\tilde{c}\, 
 \tilde{\mathbf z}\,  \tilde{\mathbf z}^{T}$ with $\tilde{z} \in \overline{\mathbb R}$ and with $0 < \tilde{c} \le \Big( \sum\limits_{j=0}^{N-1} \frac{(\tilde{\mathbf z}^{T} {\mathbf v}_{j})^2}{\lambda_{0} + \lambda_{j}}\Big)^{-1}$ solves $(\ref{probs})$ with
 the optimal error $\| {\mathbf A} - \tilde{c}\, 
 \tilde{\mathbf z}\,  \tilde{\mathbf z}^{T}\|_{2} = \| {\mathbf A} \|_{2}$.\\
 If there exist eigenvalues $\lambda_{0}$, $\lambda_{1}$ with $\lambda_{0} = \| {\mathbf A}\|_{2} = |\lambda_{1}|$ and $\lambda_1= -\lambda_0$ then there may be no real rank-1 Hankel matrix that satisfies $\| {\mathbf A} - \tilde{c}\, 
 \tilde{\mathbf z}\,  \tilde{\mathbf z}^{T}\|_{2} = \| {\mathbf A} \|_{2}$.
\end{corollary} 

\begin{proof}
We use the same notations as before for eigenvalues and eigenvectors of ${\mathbf A}$, in particular ${\mathbf A} = {\mathbf V} {\mathbf \Lambda} {\mathbf V}^{T}$ and $\mubf = {\mathbf V}^{T} \tilde{\mathbf z}$. 
The Eckart-Young-Mirsky Theorem  implies that  the optimal error of an unstructured  rank-1 approximation of ${\mathbf A}$  is $\lambda_{0} = \tilde{\lambda} = \|{\mathbf A}\|_{2}$ in the considered case.  
This error is also achieved by the rank-1 Hankel approximation, if the two matrices ${\mathbf M}_{1}(\tilde{\lambda})$ and ${\mathbf M}_{2}(\tilde{\lambda})$ in (\ref{m1}) and (\ref{m2}) are  both positive semidefinite.
With $\tilde{\lambda} = \lambda_{0}$  we observe that $\tilde{\lambda} {\mathbf I}_{N} - {\mathbf \Lambda}$  as well as $\tilde{c} \, \tilde {\mubf} {\mubf}^{T} $  are both positive semidefinite  for all $\tilde{z} \in \overline{\mathbb R}$ and $\tilde{c} \ge 0$, and thus ${\mathbf M}_{1}(\tilde{\lambda})$ is positive semidefinite.

Further, ${\mathbf M}_{2}(\tilde{\lambda})$  is the difference  of two  positive semidefinite matrices $\tilde{\lambda} {\mathbf I}_{N} + {\mathbf \Lambda}$  and $\tilde{c} \, \tilde {\mubf} {\mubf}^{T} $, and Lemma \ref{lem3} implies  that ${\mathbf M}_{2}(\tilde{\lambda})$ is positive semidefinite if and only if the  condition 
\begin{equation}\label{dop} \sumprime_{k=0}^{N-1} \frac{\mu_{k}^{2}}{\lambda_{0} + \lambda_{k}} \le \frac{1}{\tilde{c}}
\end{equation}
is satisfied.
If all eigenvalues $\lambda_{j}$ of ${\mathbf A}$  with $|\lambda_{j}| = \lambda_{0} = \| {\mathbf A}\|_{2}$  have the same sign, then $\lambda_{0} + \lambda_{k} >0$ for all $k=0, \ldots , N-1$ in the sum above, and we find for every $\tilde{z} \in \overline{\mathbb R}$ the range 
$$  
0 < \tilde{c} \le \Big(\sumprime_{k=0}^{N-1} \frac{\mu_{k}^{2}}{\lambda_{0} + \lambda_{k}} \Big)^{-1}
$$
such that $\| {\mathbf A} - \tilde{c}\, 
 \tilde{\mathbf z}\,  \tilde{\mathbf z}^{T}\|_{2} = \| {\mathbf A} \|_{2}$.
If there are eigenvalues $\lambda_{j}$ 
of ${\mathbf A}$  with $|\lambda_{j}| = \lambda_{0} = \| {\mathbf A}\|_{2}$ of different sign, then we obtain  at least one vanishing denomiantor in 
the sum in (\ref{dop}), and the condition of positive semidefiniteness for 
${\mathbf M}_{2}(\tilde{\lambda})$ is only satisfied for values 
$\tilde{z}$ with vanishing corresponding numerators.
Already for the case that the 
singular value $\lambda_{0}$ is of multiplicity $2$ with 
$\lambda_{0}=-\lambda_{1}$, we may not find such  value 
$\tilde{z}$, if for example $\mu_{0}={\mathbf v}_{0}^{T} {\mathbf z}$ and $\mu_{1}={\mathbf v}_{1}^{T} {\mathbf z}$ have no real zeros. 
In this case, there is no meaningful solution of problem (\ref{probs}) of true rank $1$, but only the trivial solution of a zero-matrix.
\end{proof}

\subsection{Computation of the Optimal Rank-1 Hankel Approximation for the Spectral Norm}
\label{last}
The conditions shown in \cref{theospectral1} can be used to provide an algorithm for computing the optimal rank-1 Hankel approximation numerically. 
First, we can verify whether \cref{best} can be satisfied. If this is possible for some $\tilde{z}$, we can choose $\tilde{c}$ according to (\ref{cc}).
If there is no $\tilde{z} \in \overline{{\mathbb R}}$ satisfying \cref{best}, then we have to employ the relations \cref{spectral2} and \cref{spectral3} in \cref{theospectral1} to determine $\tilde z$ and $\tilde c$. We use the following observation.

For fixed $\lambda^2\in(\lambda_1^2,\lambda_0^2)$ define $f_\lambda(z) \coloneqq f(z,\lambda^2)$. Since for fixed $z$, $f(z,\lambda^2)$ is strictly monotonically increasing in $\lambda^2$, formula \cref{spectral2} implies:\\ 
If $\max_{z} f_{\lambda}(z) >0$, then the optimal error $\tilde{\lambda}$ in \cref{la} satisfies $\tilde{\lambda}^{2} < \lambda^{2}$.\\
If $\max_{z} f_{\lambda}(z) <0$, then the optimal error satisfies $\tilde{\lambda}^{2} > \lambda^{2}$.\\
If $\max_{z} f_{\lambda}(z) =0$, then the optimal error satisfies $\tilde{\lambda}^{2} = \lambda^{2}$ and the rank-1 Hankel approximation is generated by this zero $\tilde{z} \in\mathop{\mathrm{argmax}}\limits_{z \in \overline{\mathbb R}} f_\lambda(z)$ and $\tilde{c}$ from \cref{spectral3}.

To find a simple range, where we have to search for the maximum of  $f_{\lambda}$,  we  apply  an observation similar to that used in \cref{theoneu} for the Frobenius norm.
Let
\begin{equation*}
    f^{(1)}_{\lambda}(z) \coloneqq \sum_{j=0}^{N-1} \frac{({\mathbf v}_{j}^{T} {\mathbf J}_{N} {\mathbf z})^{2}}{\lambda_{j}^{2} - \lambda^{2}},
\end{equation*}
where ${\mathbf J}_{N}$ denotes the counter identity in \cref{J}.
Then we observe  for $z\neq 0$  because of   ${\mathbf z}(1/z) =    {\mathbf J}_{N} \, {\mathbf z}(z) = {\mathbf J}_{N} \, {\mathbf z}$ that
\begin{equation*}
    f_{\lambda}\left(\frac{1}{z}\right)
    =  \sum_{j=0}^{N-1} \frac{\big({\mathbf v}_{j}^{T}  {\mathbf z}(1/z)\big)^{2}}{\lambda_{j}^{2} - \lambda^{2}}
    =  \sum_{j=0}^{N-1} \frac{\big({\mathbf v}_{j}^{T} {\mathbf J}_{N} {\mathbf z}(z)\big)^{2}}{\lambda_{j}^{2} - \lambda^{2}}
    = f^{(1)}_{\lambda}(z).
\end{equation*}
In particular, $f_{\lambda}^{(1)}(0) = \lim_{z \to \infty}f_{\lambda}(z)$.
Thus, we only have to search for the maximum of $f_{\lambda}(z)$ in the interval $[-1,1]$ and  for the maximum of  $f^{(1)}_{\lambda}(z)$  in  $(-1, 1)$. We obtain \cref{alg1} to compute the optimal rank-1 Hankel approximation of ${\mathbf A}$ with respect to the spectral norm as well as the corresponding error.

\begin{algorithm}[Optimal rank-1 Hankel approximation w.r.t.\ the spectral norm]
\label{alg1}
\textbf{Input:} Symmetric matrix ${\mathbf A} \in {\mathbb R}^{N \times N}$ with single largest singular value $\lambda_{0}>0$, threshold $\epsilon>0$.

\begin{enumerate}
    \item Compute the SVD of ${\mathbf A}$ to obtain the singular values  $\lambda_{0} > |\lambda_{1}| \ge \ldots \ge |\lambda_{N-1}| \ge 0$ and the normalized eigenvectors ${\mathbf v}_{0}, \ldots , {\mathbf v}_{N-1}$, such that ${\mathbf V} = ({\mathbf v}_{0} \ldots  {\mathbf v}_{N-1})$ is an orthogonal matrix.
    \item Compute the set $\Sigma$ of joint real zeros of the functions $v_{j}(z)\coloneqq {\mathbf v}_{j}^{T} {\mathbf z}$ corresponding to eigenvalues $\lambda_{j}$ with $|\lambda_{j}|= |\lambda_{1}|$ and with ${\mathbf z}$ as in \cref{z} or ${\mathbf z} = {\mathbf e}_{N}$. 
    For each $z \in \Sigma$ compute 
    \begin{equation*}
          f_{\lambda_{1}}(z)
        =  \sumprime_{j=0}^{N-1} \frac{({\mathbf v}_{j}^{T} {\mathbf z})^{2}}{\lambda_{j}^{2} - \lambda_{1}^{2}}.
    \end{equation*}
    If a value $z \in \Sigma$ satisfies $f_{\lambda_{1}}(z) \ge 0$ then set
    \begin{equation*}
        \tilde{\lambda} \coloneqq |\lambda_{1}|,
        \qquad \tilde{z} \coloneqq z,
        \qquad  \tilde{c} \coloneqq \left( \sumprime_{j=0}^{N-1} \frac{({\mathbf v}_{j}^{T} \tilde{\mathbf z})^{2}}{\lambda_{j}- |\lambda_{1}|} \right)^{-1}.
    \end{equation*}
    \item If $\Sigma = \emptyset$ or if there is no $z\in \Sigma$ satisfying $f_{\lambda_{1}}(z) \ge 0$ then apply the following bisection iteration:\\
    Set $a\coloneqq|\lambda_{1}|$ and $b\coloneqq\lambda_{0}$.
    While $b-a >\epsilon$ iterate:
    \begin{enumerate}
        \item Compute $x\coloneqq\frac{a + b}{2}$. 
        Find the maximal value $W$ of 
        \begin{equation*}
           f_{x}(z) = \sum_{j=0}^{N-1} \frac{({\mathbf v}_{j}^{T} {\mathbf z})^{2}}{\lambda_{j}^{2} - x^{2}} \qquad \text{for $z\in [-1, 1]$}
        \end{equation*}
        and of
        \begin{equation*}
            f^{(1)}_{x}(z) = \sum_{j=0}^{N-1} \frac{({\mathbf v}_{j}^{T} {\mathbf J}_{N}{\mathbf z})^{2}}{\lambda_{j}^{2} - x^{2}} \qquad \text{for $z \in (-1,1)$}.
        \end{equation*}
        \item If $W=0$, then we have found the optimal solution, go to 4.\\
        If $W >0$, then $b\coloneqq x$, else $a\coloneqq x$.
    \end{enumerate}
    \item Set $\tilde{\lambda} \coloneqq {x}$. If $f_{\tilde{\lambda}}$ possesses a zero $z$ in $[-1, 1]$ then $\tilde{z}\coloneqq z$, otherwise, if  $f^{(1)}_{\tilde{\lambda}}$ possesses a zero $z$ in $(-1,1)$ then set $\tilde{z}\coloneqq1/z$. Compute
    \begin{equation*}
        \tilde{c} \coloneqq \left( \sum_{j=0}^{N-1} \frac{({\mathbf v}_{j}^{T} \tilde{\mathbf z})^{2}}{\lambda_{j}- \tilde{\lambda}}\right)^{-1}.
    \end{equation*}
\end{enumerate}
\textbf{Output:} $\tilde{z}$, $\tilde{c}$ generating an optimal rank-1 Hankel approximation of ${\mathbf A}$  with respect
to the spectral norm,
error $\tilde{\lambda} = \| {\mathbf A} - \tilde{c} \, \tilde{\mathbf z}\tilde{\mathbf z}^{T} \|_{2}$.
\end{algorithm}

\begin{remark}
Obviously, the optimal rank-1 Hankel approximation depends on the distribution of all eigenvalues of ${\mathbf A}$ as well as on the structure of the eigenvectors of ${\mathbf A}$. In particular, the optimal parameters $\tilde{z}$ and $\tilde{c}$ generating the optimal rank-1 Hankel approximation of ${\mathbf A}$ with regard to the spectral norm usually do not coincide with those parameters found for the Frobenius norm.
\end{remark}

\begin{example}
We consider the Hankel matrix 
\begin{equation*}
    {\mathbf A} \coloneqq \begin{pmatrix} 
                            3 & 2 & 1 & 1 \\
                            2 & 1 & 1 & 2 \\
                            1 & 1 & 2 & 5 \\
                            1 & 2 &  5 & 2
                        \end{pmatrix}
\end{equation*}
with the eigenvalues (rounded to 6 digits)
\begin{equation*}
    \lambda_{0}= 8.421093, \, \lambda_{1}= -3.155074, \, \lambda_{2}= 3.009151, \, \lambda_{3}= -0.275170.
\end{equation*}
With \cref{theo1}, for the optimal rank-1 Hankel approximation with regard to the Frobenius norm, we obtain the parameters 
\begin{equation}\label{bsp}
    \tilde{z} = 1.225640, \qquad \tilde{c} = 2.912647, 
\end{equation}
and the error $\|{\mathbf A} - \tilde{c} \, \tilde{\mathbf z} \, {\mathbf z}^{T}\|_{F} = 4.568510$. The spectral norm of the obtained matrix 
is $\|{\mathbf A} - \tilde{c} \, \tilde{\mathbf z} \, {\mathbf z}^{T}\|_{2} = 3.208509 $.\\
Now we consider the rank-1 Hankel approximation with regard to the spectral norm.
In this example, the polynomial $v_{1}(z)\coloneqq{\mathbf v}_{1}^{T}{\mathbf z}$ possesses three real zeros at $z_{1}= -0.391861$, $z_{2}= 0.193813$, and $z_{3}= 1.126551$. 
At these points, we find
\begin{equation*}
    f(z_{1},\lambda_{1}^{2}) = -0.455125, \quad f(z_{2},\lambda_{1}^{2}) = -0.808914, \quad f(z_{3},\lambda_{1}^{2}) = -0.002521.
\end{equation*}
Therefore, we cannot achieve the error $|\lambda_{1}|= 3.155074$.
\Cref{alg1} provides the optimal parameters 
\begin{equation*}
    \tilde{z} = 1.143122, \qquad \tilde{c} =  3.986514 ,
\end{equation*}
and we obtain the error $\|{\mathbf A} - \tilde{c} \, \tilde{\mathbf z} \, {\mathbf z}^{T}\|_{2} = 3.159482$. At the same time, for these parameters we get the Frobenius norm $\|{\mathbf A} - \tilde{c} \, \tilde{\mathbf z} \, {\mathbf z}^{T}\|_{F} = 4.932743$. \\
For comparison, the Cadzow algorithm (considered in the next section) provides, after 15 iterations the parameters $z= 1.252213$ and $c=2.791631$ and achieves the error norms $\|{\mathbf A} - \tilde{c} \, \tilde{\mathbf z} \, {\mathbf z}^{T}\|_{2} = 3.239722$ and  $\|{\mathbf A} - \tilde{c} \, \tilde{\mathbf z} \, {\mathbf z}^{T}\|_{F} = 4.574811$.
\end{example}

\begin{remark}
The AAK theory for infinite Hankel matrices tells us, that the optimal parameter $\tilde{z}$ should be a zero of the Laurent polynomial obtained from the  (infinite) eigenvector corresponding to the second singular value $\sigma_{1}$, see e.g.\, \cite{Beylkin05, PP16}. Transferred to our case of finite matrices, we have to inspect all zeros of  $v_{1}(z)={\mathbf v}_{1}^{T} {\mathbf z}$. This is exactly, what we are doing already, when we want to check, whether the error known from the unstructured case can be achieved, see \cref{alg2}, step 2. 
As we have seen in the example above, none of the zeros of $v_{1}(z)$ provides the optimal parameter, but $z_{3}=1.126551$ is close to $\tilde{z}$ in \cref{bsp}. We refer to \cite{Beylkin05} for further error estimates. 
\end{remark}

\section{Rank-1 Hankel Approximation Using the Cadzow Algorithm}
\label{sec:cad}
\setcounter{equation}{0}

Finally, in this section we will consider the Cadzow algorithm. We will show, that the Cadzow iteration for the rank-1 Hankel approximation always converges to a fixed point. 
In accordance with \cite{deMoor94}, we will also see that the obtained result is usually not optimal with regard to the Frobenius norm or the spectral norm. Note that  the general results on convergence of alternating projections  on manifolds in \cite{Lewis} and \cite{Andersson13}  cannot be applied in this case, see also Remark \ref{remlast}. \\
We use the definition of the orthogonal projection onto the linear space of Hankel matrices given in \cref{eq:H-proj-1,eq:H-proj-2}.
Then the Cadzow algorithm can be stated as follows.
\begin{algorithm}[Cadzow algorithm for rank-1 Hankel approximation]
\label{alg2}
    \textbf{Input:} ${\mathbf A} \in {\mathbb C}^{M \times N}$ with $\rank {\mathbf A} \ge 1$ and single largest singular value.
    \begin{enumerate}
        \item Compute the largest singular value $\sigma_{0}$ of ${\mathbf A}$ and the corresponding normalized singular vectors ${\mathbf u}_{0}$, ${\mathbf v}_{0}$, such that 
        \begin{equation*}
            {\mathbf A}_{0} \coloneqq \sigma_{0} \, {\mathbf u}_{0} \, {\mathbf v}_{0}^{*} 
                    \end{equation*}
        is the best (unstructured) rank-1 approximation of ${\mathbf A}$.         
 \item For $j=1,2,\dots$ do
        \begin{enumerate} 
            \item $\tilde{\mathbf A}_{j}\coloneqq P({\mathbf A}_{j-1})$, where $P$ is given in \cref{eq:H-proj-1,eq:H-proj-2}.
            \item Compute the optimal (unstructured) rank-1 approximation of $\tilde{\mathbf A}_{j}$, 
            \begin{equation*}
                {\mathbf A}_{j} \coloneqq \sigma_{j} \, {\mathbf u}_{j} \, {\mathbf v}_{j}^{*}, 
            \end{equation*}
            where $\sigma_{j}$ is the largest singular value of $\tilde{\mathbf A}_{j}$ with normalized singular vectors ${\mathbf u}_{j}$, ${\mathbf v}_{j}$.
        \end{enumerate}
    \end{enumerate}
    \textbf{Output:} $\;\, {\mathbf A}_{\infty} = {\mathbf 0}$ if ${\sigma} \coloneqq \lim\limits_{j\to \infty} \sigma_{j} = 0$ or \\
     \phantom{Output:} $\;\; \; \, {\mathbf A}_{\infty} = \sigma \, {\mathbf u} {\mathbf v}^{*}$  if ${\sigma} \coloneqq \lim\limits_{j\to \infty} \sigma_{j} >0$, where
    $\displaystyle{\mathbf u} {\mathbf v}^{*}\coloneqq \lim\limits_{j\to \infty} {\mathbf v}_{j} {\mathbf v}_{j}^{*}$.
           
\end{algorithm}

If the rank-1 approximation of $\tilde{\mathbf A}_{j}$ in step 2 b) is not unique, then we take just the first singular vectors that are given by the used SVD algorithm. Note that floating point precision errors in numerical algorithms usually prevent such occasions. 
As we will see, the Cadzow Algorithm \ref{alg2} can be understood as an alternating projection algorithm. In case of convergence, we usually obtain a rank-1 Hankel approximation $\sigma \, {\mathbf u} \, {\mathbf v}^{*}$ of ${\mathbf A}$. We will show convergence of \cref{alg2} to a unique fixed point, which is either the zero matrix (and thus no rank-1 approximation) or a rank-1 Hankel matrix.
To analyse the convergence properties of \cref{alg2}, we start with the following lemma.

\begin{lemma} \label{lem4}
Let $2 \le M \le N$  and ${\mathbf A} \in {\mathbb C}^{M \times N}$.
Then the projection $P({\mathbf A})$ in \cref{eq:H-proj-1,eq:H-proj-2} satisfies
\begin{equation*}
\| P({\mathbf A})\|_{F} \le \| {\mathbf A} \|_{F},
\end{equation*} 
and equality holds if and only if ${\mathbf A}$ is a Hankel matrix. Moreover, if ${\mathbf A}= {\mathbf a} \, {\mathbf b}^{*}$  
with  ${\mathbf a} \in {\mathbb C}^{M}$ and 
${\mathbf b}\in {\mathbb C}^{N}$, then 
\begin{equation}\label{P1}
    \|P({\mathbf a}\,  {\mathbf b}^{*})\|_{2} 
    \le \|P({\mathbf a} \, {\mathbf b}^{*})\|_{F} 
    \le \|{\mathbf a} \, {\mathbf b}^{*}\|_{F} 
    = \|{\mathbf a} \, {\mathbf b}^{*}\|_{2} 
    = \|{\mathbf a}\|_{2} \, \|{\mathbf b}\|_{2},
\end{equation}
and the equalities $ \|P({\mathbf a} \, {\mathbf b}^{*})\|_{F} = \|{\mathbf a} \, {\mathbf b}^{*}\|_{F} $ and 
$ \|P({\mathbf a} \, {\mathbf b}^{*})\|_{2} = \|{\mathbf a} \, {\mathbf b}^{*}\|_{2} $ hold, if and only if there exists $z\in {\mathbb C}$ such that 
${\mathbf a} = {\mathbf z}_{M}$  and ${\mathbf b} = \overline{\mathbf z}_{N}$ as given in $(\ref{z})$ or ${\mathbf a}={\mathbf e}_{M}$ and 
${\mathbf b}={\mathbf e}_{N}$ as given in \cref{en}.
\end{lemma}

\begin{proof}
For ${\mathbf A} \in {\mathbb C}^{M \times N}$ with $M \le N$ we define the vectorization by going through the antidiagonals of ${\mathbf A}$, 
$$ \vect({\mathbf A}) := \left( \begin{array}{c} a_{0,0} \\
(a_{j,1-j})_{j=0}^{1} \\
(a_{j,2-j})_{j=0}^{2} \\
\vdots \\
a_{M-1, N-1} \end{array} \right) \in {\mathbb C}^{MN}. $$
Then the Hankel projection $P({\mathbf A})$ in \cref{eq:H-proj-1,eq:H-proj-2} can be rewritten as the linear mapping
$$ \vect(P({\mathbf A})) = {\mathbf P} \, \vect({\mathbf A}), $$
where  ${\mathbf P}$ is a block diagonal matrix of the form 
\begin{equation}\label{PP} {\mathbf P} := \left( \begin{array}{ccccccccc}
{\mathbf E}_{1} & & & & & & & & \\
 & \frac{1}{2} {\mathbf E}_{2}& &&&&&& \\
 && \ddots &&&&&&\\
 &&& \frac{1}{M} {\mathbf E}_{M} &&&&& \\
 &&&& \ddots &&&& \\
 &&&&& \frac{1}{M} {\mathbf E}_{M} &&& \\
 &&&&&& \ddots && \\
 &&&&&&& \frac{1}{2} {\mathbf E}_{2} & \\
 &&&&&&&& {\mathbf E}_{1} \end{array} \right) \in {\mathbb C}^{MN \times MN}.
\end{equation}
Here, ${\mathbf E}_{n}\coloneqq (1 )_{j,k=0}^{n-1}$ is an $n\times n$ square matrix containing only ones, and the block $\frac{1}{M} {\mathbf E}_{M}$ occurs $N-M+1$ times.
 Obviously, $\|{\mathbf P}\|_{2} = 1$ and ${\mathbf P}$ possesses the eigenvalue $1$ with multiplicity $M+N-1$ and the eigenvalue $0$ with multiplicity $MN-M-N+1$. Therefore, any vector ${\mathbf v} \in {\mathbb C}^{MN}$ can be written as an orthogonal sum ${\mathbf v} = {\mathbf v}_{1} \oplus {\mathbf v}_{2}$  with ${\mathbf P} {\mathbf v}_{1} = {\mathbf v}_{1}$, ${\mathbf P}
 {\mathbf v}_{2} = {\mathbf 0}$, and ${\mathbf v}_{1}^{*} {\mathbf v}_{2} = 0$.
 In particular, we have
 $$ \|P({\mathbf A})\|_{2} \le \|P({\mathbf A})\|_{F} = \|{\mathbf P} \vect({\mathbf A}) \|_{2} \le \|{\mathbf P}\|_{2} \, \|\vect({\mathbf A}) \|_{2} =\|{\mathbf A}\|_{F},$$
 and equality only holds if $\vect({\mathbf A}) = {\mathbf P} \vect({\mathbf A})$, i.e., if ${\mathbf A}$ has Hankel structure.
 If ${\mathbf A}= {\mathbf a} \, {\mathbf b}^{*}$, then Lemma \ref{thm:rank-1_hankel} implies that ${\mathbf P} \, \vect({\mathbf a} \, {\mathbf b}^{*}) = \vect({\mathbf a} \, {\mathbf b}^{*})$ is only true if and only if ${\mathbf a} = {\mathbf z}_{M}$, ${\mathbf b} = \overline{\mathbf z}_{N}$  or ${\mathbf a}={\mathbf e}_{M}$, 
${\mathbf b}={\mathbf e}_{N}$, while $\|{\mathbf a} \, {\mathbf b}^{*}\|_{F}  = \|{\mathbf a} \, {\mathbf b}^{*}\|_{2} = \|{\mathbf a}\|_{2} \, \|{\mathbf b}\|_{2}$ is obvious.
\end{proof}

Since the map from $\tilde{\mathbf A}_j$ onto its optimal rank-1 approximation ${\mathbf  A}_j = \sigma_j \, {\mathbf u}_j \, {\mathbf v}_j^*$ in \cref{alg2} is an orthogonal projection onto the manifold  of rank-1 matrices of size $M \times N$, the Cadzow algorithm is indeed an alternating projection algorithm.
Next we show that there always exists an accumulation point $\sigma \, {\mathbf u} {\mathbf v}^{*}$  of the sequence $(\sigma_{j}\, {\mathbf u}_{j} {\mathbf v}_{j})_{j=0}^{\infty}$ computed in Algorithm \ref{alg2}, which is a fixed point, namely either  a rank-1 Hankel matrix or the zero matrix.

\begin{theorem}\label{lemcad1}
Let ${\mathbf A} \in {\mathbb C}^{M \times N}$ with $2 \le M \le N$ and $\rank({\mathbf A}) \ge 1$.
Then the sequence $(\sigma_{j})_{j=0}^{\infty}$ of singular values  
in the Cadzow \cref{alg2} converges.\\
If $\sigma = \lim\limits_{j \to \infty} \sigma_{j}=0$, then $\cref{alg2}$ converges to the zero matrix. \\
If $\sigma = \lim\limits_{j \to \infty} \sigma_{j}>0$, then there exists a subsequence $({\mathbf u}_{j_{\ell}} {\mathbf v}_{j_{\ell}}^{*})_{\ell=0}^{\infty}$  of  $({\mathbf u}_{j} {\mathbf v}_{j}^{*})_{j=0}^{\infty}$ in $\cref{alg2}$ that converges to a limit ${\mathbf u} {\mathbf v}^{*}$,
and $\sigma {\mathbf u} {\mathbf v}^{*}$ is a rank-1 Hankel matrix, i.e.,  there exists  $z \in {\mathbb C}$ such that 
\begin{equation*}
    {\mathbf u} {\mathbf v}^{*} \coloneqq \lim\limits_{\ell\to \infty} {\mathbf u}_{j_{\ell}} {\mathbf v}_{j_{\ell}}^{*}=
    {\mathbf z}_{M} {\mathbf z}_{N}^{T}
\end{equation*}
with ${\mathbf z}_{M}$ and ${\mathbf z}_{N}$ as in \cref{z}, or 
${\mathbf u} {\mathbf v}^{*} = {\mathbf e}_{M} {\mathbf e}_{N}^{T}$.
\end{theorem}
\begin{proof}
1.\ If the first singular vectors $\mathbf u_0$ and $\mathbf v_0$ of $\mathbf A$ are of the form $\mathbf z_M$ and $\mathbf z_N$ or $\mathbf e_M$ and $\mathbf e_N$, respectively, then the optimal rank-1 approximation of $\mathbf A$ already has Hankel structure. Therefore, by definition of $P$, we have
    $P(\sigma_{0} {\mathbf u}_{0} \, {\mathbf v}_{0}^{*}) = \sigma_{0} {\mathbf u}_{0} \, {\mathbf v}_{0}^{*}$,
and the algorithm immediately stops, since we find constant sequences $(\sigma_{j})_{j=0}^{\infty}$  and $({\mathbf u}_{j} {\mathbf v}_{j}^{*})_{j=0}^{\infty}$.

2.\ Assume now that ${\mathbf u}_{0}{\mathbf v}_{0}^{*}$ is neither of the form 
 ${\mathbf z}_{M}{\mathbf z}_{N}^{T}$ for some $z \in {\mathbb C}$ nor  ${\mathbf e}_{M}{\mathbf e}_{N}^{T}$.
Then, by \cref{lem4}, we find  for the largest singular value of $\tilde{\mathbf A}_{1} = P(\sigma_{0} \, {\mathbf u}_{0} \,{\mathbf v}_{0}^{*})$
\begin{equation*}
    \sigma_{1} = \|P(\sigma_{0} \, {\mathbf u}_{0} \, {\mathbf v}_{0}^{*})\|_{2} < \sigma_{0} \|{\mathbf u}_{0} \, {\mathbf v}_{0}^{*}\|_{2} = \sigma_{0}.
\end{equation*}
For any $j\ge 1$ we obtain analogously
\begin{equation}\label{sigma}
    \sigma_{j+1} = \|P(\sigma_{j} \, {\mathbf u}_{j} \, {\mathbf v}_{j}^{*})\|_{2} \le \sigma_{j} \|{\mathbf u}_{j} \, {\mathbf v}_{j}^{*}\|_{2} = \sigma_{j},
\end{equation}
and this inequality is strict as long as $\mathbf u_j \mathbf v_j^*$ does not have Hankel structure (see \cref{lem4}).
Thus, the sequence of singular values $(\sigma_{j})_{j=0}^{\infty}$ decreases monotonically. Since $\sigma_{j}\ge 0$ for all $j$, convergence follows, and we write $\sigma\coloneqq \lim_{j \to \infty} \sigma_{j}$.

3.\ 
If $\sigma=\lim\limits_{j \to \infty} \sigma_{j}=0$, then $(\sigma_{j} {\mathbf u}_{j} {\mathbf v}_{j}^{*})_{j=0}^{\infty}$ in  Algorithm \ref{alg2} converges to the zero matrix,  i.e., it fails to converge to a rank-1 Hankel matrix. In this case ${\mathbf u}_{j} {\mathbf v}_{j}^{*}$ may not converge to a matrix of Hankel structure.

Let now $\sigma=\lim\limits_{j \to \infty} \sigma_{j} = \lim\limits_{j \to \infty} \| P(\sigma_{j-1}{\mathbf u}_{j-1} \, {\mathbf v}_{j-1}^{*}) \|_{2} = \lim\limits_{j \to \infty} \sigma_{j}  \, \lim\limits_{j \to \infty}  \|
P({\mathbf u}_{j} \, {\mathbf v}_{j}^{*}) \|_{2} >0$. Thus, 
\begin{equation}\label{sigma1}  \lim_{j \to \infty} \|  P({\mathbf u}_{j} \, 
{\mathbf v}_{j}^{*}) \|_{2} = 1 = \|  {\mathbf u}_{j} \, 
{\mathbf v}_{j}^{*} \|_{2} .
\end{equation}
Since the vectors ${\mathbf u}_{j}$ and ${\mathbf v}_{j}$ are normalized and therefore the sequence of matrices $\mathbf u_j \mathbf v_j^*$ is bounded, we conclude that there exists a subsequence $({\mathbf u}_{j_{\ell}} \, {\mathbf v}_{j_{\ell}}^{*})_{\ell=0}^{\infty}$ that converges to an accumulation point ${\mathbf u} {\mathbf v}^*$, which is by (\ref{sigma1}) and Lemma \ref{lem4} a fixed point of the Cadzow iteration, i.e., $P({\mathbf u} {\mathbf v}^*) = {\mathbf u} {\mathbf v}^*$.
\end{proof}

Note that in \cite{Zvonarev17}  a similar result has been shown  for low-rank Hankel approximation by the Cadzow algorithm. But \cite{Zvonarev17} did not study the question, whether the partial sequence indeed converges to  a matrix with the desired rank.
In the remainder of this section we will show that in fact the full sequence $({\mathbf u}_{j} {\mathbf v}_{j}^{*})_{j=0}^{\infty}$ in Algorithm \ref{alg2} converges to the found fixed point $\sigma {\mathbf u} {\mathbf v}^*$. 

Our proof is based on the observation that a rank-1 matrix ${\mathbf a} {\mathbf b}^{*} \in {\mathbb C}^{M \times N}$ which is close to the subspace of Hankel matrices, is also close to the manifold of rank-1 Hankel matrices.
\begin{lemma}\label{lemcad2}
For ${\mathbf a}= (a_{0}, \ldots , a_{M-1})^{T} \in {\mathbb C}^{M}$ and ${\mathbf b} = (b_{0}, \ldots , b_{N-1})^{T} \in {\mathbb C}^{N}$  with $\|{\mathbf a}\|_{2} = \| {\mathbf b}\|_{2} = 1$ and 
\begin{equation}\label{d1} \|{\mathbf a} {\mathbf b}^{*} - P({\mathbf a} {\mathbf b}^{*}) \|_{\infty} \le \delta, 
\end{equation}
where $\| \cdot \|_{\infty}$ is componentwise maximum norm as given in $(\ref{norm1})$,  we have 
$$ \min_{z \in \overline{\mathbb C}, c \in {\mathbb C}} \|  {\mathbf a} {\mathbf b}^{*} -  c \, {\mathbf z}_{M} \, {\mathbf z}_{N}^{T} \|_{F}
< C \delta, $$
where the constant $C$ only depends on the dimensions $M$ and $N$.
\end{lemma}

\begin{proof}
Let $a_{j}$ be the by modulus largest component of ${\mathbf a}$ and ${b}_{k}$ the by modulus largest component of ${\mathbf b}$.
Then $|a_{j}| \ge \frac{1}{M}$ and $|b_{k}| \ge \frac{1}{N}$. Assume that $j< M-1$, otherwise we consider ${\mathbf J}_{M} {\mathbf a} {\mathbf b}^{*} {\mathbf J}_{N}$ instead of ${\mathbf a}{\mathbf b}^{*}$.
We choose $z:=\frac{a_{j+1}}{a_{j}}$. Then we obtain for any $\ell=1, \ldots , N-1$, from the assumption (\ref{d1})
$$ a_{j} \overline{b}_{\ell} = a_{j+1} \overline{b}_{\ell-1} + \delta_{j, \ell} = z \, a_{j} \overline{b}_{\ell-1} + \delta_{j, \ell} $$
with  $|\delta_{j, \ell}| < 2\delta$, i.e., $\overline{b}_{\ell} = z \,  \overline{b}_{\ell-1} + \frac{\delta_{j, \ell}}{a_{j}}$.
Inductively, it follows that 
$$ \overline{b}_{\ell} = z^{\ell} \overline{b}_{0} + \frac{1}{a_{j}} \sum_{\nu=0}^{\ell-1} \delta_{j,\ell-\nu} z^{\nu} $$
for $\ell=1, \ldots , N-1$, and therefore 
\begin{equation}\label{v11}
    \| \overline{\mathbf b} - \overline{b}_{0} (z^{\ell})_{\ell=0}^{N-1} \|_{\infty}  = \left\| \left( \frac{1}{a_{j}} \sum_{\nu=0}^{\ell-1} \delta_{j,\ell-\nu} z^{\nu} \right)_{\ell=0}^{N-1} \right\|_{\infty}  
    \le 2 \, M\, N \, \delta, 
\end{equation}
since $|z| \le 1$ by construction. For $\ell=0$ the sum above is empty and the componentwise error vanishes. Similarly, if $k < N-1$, we find for all $\ell=1, \ldots , M-1$,
\begin{equation*}
     a_{\ell \overline{b}_{k}} = a_{\ell-1} \overline{b}_{k+1}  + \delta_{k,\ell} = a_{\ell-1} \left(z \overline{b}_{k} + \frac{\delta_{j,k+1}}{a_{j}}\right)  + \delta_{k,\ell}
\end{equation*}
with $|\delta_{k,\ell}| \le 2 \delta$, i.e., $a_{\ell} = z a_{\ell-1} + \frac{\delta_{j,k+1} a_{\ell-1}}{a_{j} \overline{b}_{k}} + \delta_{k,\ell} = 
z a_{\ell-1} + \frac{\tilde{\delta}_{j,k+1}}{\overline{b}_{k}} + \delta_{k,\ell}$, where $|\tilde{\delta}_{j,k+1}| < 2 \delta$ since $|\frac{a_{\ell-1}}{a_{j}}| \le 1$.
As before we obtain inductively 
\begin{equation*}
    a_{\ell} = z^{\ell} a_{0} + \frac{1}{\overline{b}_{k}} \sum_{\nu=0}^{\ell-1} \tilde{\delta}_{j,\ell-\nu} z^{\nu} + \sum_{\nu=0}^{\ell-1} \delta_{k, \ell-\nu} z^\nu 
\end{equation*}
for $\ell=1, \ldots , M-1$, with some $|\tilde{\delta}_{j, \ell-\nu}| \le 2 \delta$ and $|\delta_{k, \ell-\nu}| \le 2 \delta$, and therefore 
\begin{equation}\label{u11}
    \| {\mathbf a} - {a}_{0} (z^{\ell})_{\ell=0}^{M-1} \|_{\infty} \le 2 M\, (N+1) \, \delta. 
\end{equation}
If $k=N-1$,  we can replace $\overline{b}_{k}$ by $\overline{b}_{k-1}$ using that $|a_{j+1} \overline{b}_{k-1}| \ge |a_{j} b_{k}| - \delta$ which leads to $|b_{k-1}| \ge |b_{k} - \delta/a_{j}|$ to get a similar estimate for sufficiently small $\delta$.
The inequality of Lemma \ref{lemcad2} now follows  from (\ref{v11}) and (\ref{u11}) by
\begin{align*}
& \left\| {\mathbf a} {\mathbf b}^{*} - a_{0} \overline{b}_{0} (z^{\ell})_{\ell=0}^{M-1} ((z^{\ell})_{\ell=0}^{N-1})^{T}\right\|_{F}^{2} =
\sum_{j=0}^{M-1} \sum_{k=0}^{N-1} \left|a_{j} \overline{b}_{k} - a_{0} \overline{b}_{0} z^{j+k} \right|^{2} \\
 &= \sum_{j=0}^{M-1} \sum_{k=0}^{N-1} \left|a_{j} (\overline{b}_{k} - \overline{b}_{0} z^{k}) +  \overline{b}_{0} z^{k} (a_{j} - a_{0} z^{j}) \right|^{2}
\le (2M(N+1) \delta)^{2} \sum_{j=0}^{M-1} \sum_{k=0}^{N-1} \left(|a_{j}| + |\overline{b}_{0} z^{k}|\right)^{2}\\
 &\le (2M(N+1) \delta)^{2} \sum_{j=0}^{M-1} \sum_{k=0}^{N-1} 3 \left(|a_{j}|^{2} + |b_{k}|^{2} + (2MN\delta)^{2}\right)\\
 &\le 3 (2M(N+1) \delta)^{2} (N+M+ 4(MN)^{3} \delta^{2}) < 12 M^{2}(N+1)^{2} (M+N+ 4(MN)^{3}) \delta^{2}
\end{align*} 
for $\delta \le 1$. 
Therefore the inequality is true with $C<12 M^{2}(N+1)^{2} (M+N+ 4(MN)^{3})$ for the chosen $z$ and $c=a_{0}b_{0} \left(\| (z^{\ell})_{\ell=0}^{M-1}\|_{2} \| (z^{\ell})_{\ell=0}^{N-1}\|_{2}\right)$ .
\end{proof}

With these preliminaries, we can now show our main theorem of this section on the convergence of the Cadzow algorithm to one fixed point. The proof is based on the observation, that for each $j$, all further iteration matrices ${\mathbf u}_{k}{\mathbf v}_{k}^{*}$, $k >j$, are inside the ball around the optimal approximation  of ${\mathbf u}_{j} {\mathbf v}_{j}^{*}$ in the set of rank-1 Hankel matrices 
while for $j \to \infty$ the radius of these balls tends to zero.

\begin{theorem}
Let ${\mathbf A} \in {\mathbb C}^{M \times N}$ with $2 \le M \le N$ and $\rank({\mathbf A}) \ge 1$.
Then the sequence $(\sigma_{j})_{j=0}^{\infty}$ in the Cadzow \cref{alg2} converges.\\
If $\sigma = \lim\limits_{j \to \infty} \sigma_{j}=0$, then $\lim\limits_{j\to \infty} \sigma_{j}\, {\mathbf u}_{j} {\mathbf v}_{j}^{*} = {\mathbf 0}$. \\
If $\sigma = \lim\limits_{j \to \infty} \sigma_{j}>0$, then the sequence $({\mathbf u}_{j} {\mathbf v}_{j}^{*})_{j=0}^{\infty}$ converges and there exists $z \in {\mathbb C}$ such that 
\begin{equation*}
    {\mathbf u} {\mathbf v}^{*} \coloneqq \lim\limits_{j\to \infty} {\mathbf u}_{j} {\mathbf v}_{j}^{*} =
    {\mathbf z}_{M} {\mathbf z}_{N}^{T}\end{equation*}
with ${\mathbf z}_{M}$ and ${\mathbf z}_{N}$ as in \cref{z}, or 
${\mathbf u} {\mathbf v}^{*} = {\mathbf e}_{M} {\mathbf e}_{N}^{T}$, i.e., Algorithm $\ref{alg2}$ provides the rank-1 Hankel approximation $\sigma {\mathbf u} {\mathbf v}^{*}$.
\end{theorem}

\begin{proof}
As shown in Theorem \ref{lemcad1}, we always have convergence of $(\sigma_{j})_{j=0}^{\infty}$ to a limit $\sigma$, and for $\sigma=0$, Algorithm \ref{alg2} only provides the zero matrix. Further, for  $\sigma>0$, there is a  subsequence $(\mathbf u_{j_{\ell}} \mathbf v_{j_{\ell}}^{*})_{\ell =0}^{\infty}$ that converges to ${\mathbf u}{\mathbf v}^{*}= {\mathbf z}_{M} {\mathbf z}_{N}^{T}$ (or ${\mathbf e}_{M} {\mathbf e}_{N}^{T}$),  i.e., $\sigma \, {\mathbf u} {\mathbf v}^{*}$ is a rank-1 Hankel matrix. 
We show that for $\sigma>0$ the full sequence $({\mathbf u}_{j} {\mathbf v}_{j}^{*})_{j=0}^{\infty}$  converges to ${\mathbf u} {\mathbf v}^{*}$.

1.\ 
For each $j \in {\mathbb N}$, we can apply the projection onto the subspace  of Hankel matrices as in the proof of Lemma \ref{lem4}, 
\begin{equation}\label{neu} \|P({\mathbf u}_{j} {\mathbf v}_{j}^{*})\|_{F} = \|{\mathbf P} \, \vect({\mathbf u}_{j} {\mathbf v}_{j}^{*}) \|_{2} =  \|{\mathbf P} \, ({\mathbf w}_{j}^{(0)} \oplus {\mathbf w}_{j}^{(1)})\|_{2} =  \|{\mathbf P} \, ({\mathbf w}_{j}^{(1)}) \|_{2} = \|{\mathbf w}_{j}^{(1)} \|_{2},
\end{equation}
where $\vect({\mathbf u}_{j} {\mathbf v}_{j}^{*})={\mathbf w}_{j}^{(1)} \oplus {\mathbf w}_{j}^{(0)}$ with ${\mathbf P} {\mathbf w}_{j}^{(1)} = {\mathbf w}_{j}^{(1)}$, 
${\mathbf P} {\mathbf w}_{j}^{(0)} = {\mathbf 0}$, $({\mathbf w}_{j}^{(0)})^{*} {\mathbf w}_{j}^{(1)} = 0$.
Let 
$$\delta_{j}  \coloneqq  \|{\mathbf u}_{j} {\mathbf v}_{j}^{*} - P({\mathbf u}_{j}{\mathbf v}_{j}^{*}) \|_{F} = \|{\mathbf w}_{j}^{(0)} \|_{2}.$$ 
Then
\begin{equation}\label{1z} \|P({\mathbf u}_{j} {\mathbf v}_{j}^{*})\|_{F}^{2} = \|{\mathbf u}_{j} {\mathbf v}_{j}^{*}\|_{F}^{2} -\delta_{j}^{2} = 1 - \delta_{j}^{2}
\end{equation}
and $\lim\limits_{j\to \infty} \delta_{j} =0$ and is monotonically decaying, since by (\ref{P1}), (\ref{sigma}), and (\ref{sigma1}), 
$$ \lim_{j \to \infty} \frac{\sigma_{j+1}}{\sigma_{j}} = \lim_{j \to \infty} \frac{\|P({\mathbf u}_{j} {\mathbf v}_{j}^{*}) \|_{2}}{\|{\mathbf u}_{j} {\mathbf v}_{j}^{*}\|_{2}} = \lim_{j \to \infty} 
\frac{\|P({\mathbf u}_{j} {\mathbf v}_{j}^{*}) \|_{F}}{\|{\mathbf u}_{j} {\mathbf v}_{j}^{*}\|_{F}}
= \lim_{j \to \infty} \|P({\mathbf u}_{j} {\mathbf v}_{j}^{*}) \|_{F} = 1.$$

2.\ Consider now the singular value decomposition $P({\mathbf u}_{j} {\mathbf v}_{j}^{*}) = {\mathbf U}_{j} {\mathbf D}_{j} \, {\mathbf V}_{j}^*$ with 
matrices ${\mathbf U}_{j} \in {\mathbb C}^{M \times M}$, ${\mathbf V}_{j} \in {\mathbb C}^{N \times M}$ 
satisfying ${\mathbf U}_{j}^{*} {\mathbf U}_{j} = {\mathbf V}_{j}^{*} {\mathbf V}_{j} = {\mathbf I}_{M}$ and ${\mathbf D}_{j} = \diag(s_{0}, \, s_{1}, \ldots , s_{M-1})$, where $s_{0} \ge s_{1} \ge s_{2} \ldots \ge s_{M-1}$. Note that $s_{0} = \frac{\sigma_{j+1}}{\sigma_{j}}$.
Observe that the iteration vectors ${\mathbf u}_{j+1}$ and ${\mathbf v}_{j+1}$ in the Cadzow iteration are the first columns of ${\mathbf U}_{j}$ and ${\mathbf V}_{j}$, respectively.\\
Formula (\ref{1z}) implies that $\sum\limits_{\ell=0}^{M-1} s_{\ell}^{2} = 1 - \delta_{j}^{2}$, while  the  Eckart-Young Mirsky theorem yields that  
\begin{equation}\label{neu1z}
 \sum_{\ell=1}^{M-1} s_{\ell}^2\ = \|P({\mathbf u}_{j} {\mathbf v}_{j}^{*}) - s_{0} {\mathbf u}_{j+1} {\mathbf v}_{j+1}^{*} \|_{F}^{2} < \|P({\mathbf u}_{j} {\mathbf v}_{j}^{*}) -  {\mathbf u}_{j} {\mathbf v}_{j}^{*} \|_{F}^{2} = \delta_{j}^{2},
\end{equation} 
and therefore
\begin{equation}\label{d00} 1 - 2\delta_{j}^{2} <  s_{0}^{2} < 1 -  \delta_{j}^{2}. \end{equation}
Thus, there exists $j_{0} \in {\mathbb N}$  such that for all $j> j_{0}$, the value $\sigma_{j} < \sigma_{j_{0}}$ is small enough to ensure that  $s_{0}$ is the unique largest singular value of $P({\mathbf u}_{j} {\mathbf v}_{j}^{*})$.

3.\ By Lemma \ref{lemcad2}, there exist $c_{j} \in {\mathbb C}$ and $z_{j} \in \overline{\mathbb C}$ such that 
\begin{equation*}
     \left\| {\mathbf u}_{j} {\mathbf v}_{j}^{*} -  c_{j} \, {\mathbf z}_{M}(z_{j}) \, {\mathbf z}_{N}(z_{j})^{T} \right\|_{F}
    < C \delta_{j},
\end{equation*}
where $C$ only depends on the dimensions $M$ and $N$. 
Since $\|{\mathbf u}_{j} {\mathbf v}_{j}^{*}\|_{F} = \|{\mathbf z}_{M}(z_{j}) \, {\mathbf z}_{N}(z_{j})^{T}\|_{F} = 1$, we obtain  $|1 - |c|| < C \delta_{j}$, and  it follows that the $z_{j} \in \overline{\mathbb C}$ chosen above satisfies
$$ \|  {\mathbf u}_{j} {\mathbf v}_{j}^{*} -  {\mathbf z}_{M}(z_{j}) \, {\mathbf z}_{N}(z_{j})^{T} \|_{F} < C' \delta_{j}, $$
where $C' = 2C$.
We assume that $\delta_{j} < \frac{1}{2C'^{2}}$ and  show that all further iteration matrices ${\mathbf u}_{k} {\mathbf v}_{k}^{*}$, $k \ge j$ of Algorithm \ref{alg2} also satisfy the condition $\|  {\mathbf u}_{k} {\mathbf v}_{k}^{*} -  {\mathbf z}_{M}(z_{j}) {\mathbf z}_{N}(z_{j})^{T} \|_{F} < C' \,  \delta_{j}$, i.e., all ${\mathbf u}_{k} {\mathbf v}_{k}^{*}$ are in the ball of radius $C' \delta_{j}$ around ${\mathbf z}_{M}(z_{j}) \, {\mathbf z}_{N}(z_{j})^{T}$.

First, we conclude from ${\mathbf P} \vect({\mathbf z}_{M}(z_{j}) {\mathbf z}_{N}(z_{j})^{T}) = \vect({\mathbf z}_{M}(z_{j}) {\mathbf z}_{N}(z_{j})^{T})$ that 
\begin{align*}
& \| {\mathbf u}_{j} {\mathbf v}_{j}^{*} - {\mathbf z}_{M}(z_{j}){\mathbf z}_{N}(z_{j})^{T} \|_{F}^{2} = \|\vect({\mathbf u}_{j} {\mathbf v}_{j}^{*}) - \vect({\mathbf z}_{M}(z_{j}) {\mathbf z}_{N}(z_{j})^{T})\|_{2}^{2} \\
&= \| {\mathbf P}(({\mathbf w}_{j}^{(1)})- \vect({\mathbf z}_{M}(z_{j}) {\mathbf z}_{N}(z_{j})^{T}) )\oplus  {\mathbf w}_{j}^{(0)} \|_{2}^{2} \\ 
&= \| {\mathbf P}({\mathbf w}_{j}^{(1)}) - \vect({\mathbf z}_{M}(z_{j}) {\mathbf z}_{N}(z_{j})^{T})\|_{2}^{2} + \|{\mathbf w}_{j}^{(0)} \|_{2}^{2} 
= \| P({\mathbf u}_{j} {\mathbf v}_{j}^{*}) - {\mathbf z}_{M}(z_{j}) {\mathbf z}_{N}(z_{j})^{T}\|_{F}^{2} + \delta_{j}^{2},
\end{align*}
i.e., 
\begin{equation}\label{ab1z}
\| P({\mathbf u}_{j} {\mathbf v}_{j}^{*}) - {\mathbf z}_{M}(z_{j}) {\mathbf z}_{N}(z_{j})^{T}\|_{F}^{2} <  (C' \delta_{j})^{2} -  \delta_{j}^{2}.
\end{equation}

It is sufficient to show that $\|  {\mathbf u}_{j+1} {\mathbf v}_{j+1}^{*} -  {\mathbf z}_{M}(z_{j}) {\mathbf z}_{N}(z_{j})^{T} \|_{F} < C' \,  \delta_{j}$, then the argument can be repeated for $k > j+1$.
Observe that 
\begin{align*}
& \|{\mathbf u}_{j} {\mathbf v}_{j}^{*} - {\mathbf z}_{M}(z_{j}) {\mathbf z}_{N}(z_{j})^{T}\|_{F}^{2}  = \textrm{trace} ( ({\mathbf u}_{j} {\mathbf v}_{j}^{*} - 
{\mathbf z}_{M}(z_{j}) {\mathbf z}_{N}(z_{j})^{T})^{*} ({\mathbf u}_{j} {\mathbf v}_{j}^{*} - {\mathbf z}_{M}(z_{j}) {\mathbf z}_{N}(z_{j})^{T})) \\
&= \textrm{trace} \Big( {\mathbf v}_{j} {\mathbf v}_{j}^{*} + \overline{{\mathbf z}_{N}(z_{j})} {\mathbf z}_{N}(z_{j})^{T} - ({\mathbf u}_{j}^{*} {\mathbf z}_{M}(z_{j})) {\mathbf v}_{j} {\mathbf z}_{N}(z_{j})^{T}
- ({\mathbf z}_{M}(z_{j})^{*} {\mathbf u}_{j}) \overline{{\mathbf z}_{N}(z_{j})} {\mathbf v}_{j}^{*} \Big) \\
&= 2 - 2 \, \textrm{Re}(({\mathbf u}_{j}^{*} {\mathbf z}_{M}(z_{j})) ({\mathbf v}_{j}^{T} {\mathbf z}_{N}(z_{j}))).
\end{align*}
Thus, $\|{\mathbf u}_{j} {\mathbf v}_{j}^{*} - {\mathbf z}_{M}(z_{j}) {\mathbf z}_{N}(z_{j})^{T}\|_{F}^{2} < (C' \delta_{j})^{2}$ is equivalent to $\textrm{Re}(({\mathbf u}_{j}^{*} {\mathbf z}_{M}(z_{j})) ({\mathbf v}_{j}^{T} {\mathbf z}_{N}(z_{j}))) > 1 - \frac{(C' \delta_{j})^{2}}{2}$,
and it is sufficient to show that  $\textrm{Re}(({\mathbf u}_{j+1}^{*} {\mathbf z}_{M}(z_{j})) ({\mathbf v}_{j+1}^{T} {\mathbf z}_{N}(z_{j}))) > 1 - \frac{(C' \delta_{j})^{2}}{2}$.

Using (\ref{neu1z}), (\ref{ab1z})  and the SVD matrices ${\mathbf U}_{j}$ and ${\mathbf V}_{j}$ of $P({\mathbf u}_{j} {\mathbf v}_{j}^{*})$ we obtain  with $\alphabf(z_{j}) \coloneqq {\mathbf U}_{j}^{*} {\mathbf z}_{M}(z_{j})$ and $\betabf(z_{j}) \coloneqq {\mathbf V}_j^{*} \overline{\mathbf z}_{N}(z_{j)}$
\begin{align*}
\|P({\mathbf u}_{j} {\mathbf v}_{j}^{*}) - {\mathbf z}_{M}(z_{j}) {\mathbf z}_{N}(z_{j})^{T}\|_{F}^{2} &= \| {\mathbf D}_{j} - \alphabf(z) \betabf(z)^{*} \|_{F}^{2} \\
&= 1 + \sum_{\ell=0}^{M-1} s_{\ell}^{2} - 2 \sum_{\ell=0}^{M-1} s_{\ell} \, \mathrm{Re}( \alpha_{\ell}(z_{j}) \overline{\beta_{\ell}(z_{j})}) < (C' \delta_{j})^{2} - \delta_{j}^{2}.
\end{align*}
Thus we have 
\begin{equation}\label{hahaz} 
\sum\limits_{\ell=0}^{M-1} s_{\ell} \, \mathrm{Re}( \alpha_{\ell}(z_{j}) \overline{\beta_{\ell}(z_{j})}) > 1 - \frac{(C' \delta_{j})^{2}}{2},
\end{equation}
and we need to show $\mathrm{Re}( \alpha_{0}(z_{j}) \overline{\beta_{0}(z_{j})}) \ge 1 - \frac{(C' \delta_{j})^{2}}{2}$ since $\alpha_{0}(z_{j}) = {\mathbf u}_{j+1}^{*} {\mathbf z}_{M}(z_{j})$ and ${\mathbf \beta}_{0}(z_{j}) = {\mathbf v}_{j+1}^{*} \overline{\mathbf z}_{N}(z_{j})$.

By (\ref{neu1z}) and $\Big|\sum\limits_{\ell=0}^{M-1} \textrm{Re}(\alpha_{\ell}(z_{j}) \overline{\beta_{\ell}(z_{j})})\Big| \le \sum\limits_{\ell=0}^{M-1} \Big|\alpha_{\ell}(z_{j}) \overline{\beta_{\ell}(z_{j})}\Big|
< \| {\alphabf}(z_{j}) \|_{2} \| \betabf(z_{j}) \|_{2} = 1$, we find
\begin{equation*}
    \sum_{\ell=1}^{M-1} s_{\ell} \, \textrm{Re}( \alpha_{\ell}(z_{j}) \overline{\beta_{\ell}(z_{j})}) < s_{1} \, \sum_{\ell=1}^{M-1} | \alpha_{\ell}(z_{j}) \overline{\beta_{\ell}(z_{j})}| <  \delta_{j} \,  (1- \textrm{Re}( \alpha_{0}(z_{j}) \overline{\beta_{0}(z_{j})})).
\end{equation*}
Thus,  (\ref{hahaz}) implies  that 
\begin{equation*}
    s_{0} \, \mathrm{Re}( \alpha_{0}(z_{j}) \overline{\beta_{0}(z_{j})}) + \delta_{j}(1-\mathrm{Re}( \alpha_{0}(z_{j}) \overline{\beta_{0}(z_{j})})) >  \sum_{\ell=0}^{M-1} s_{\ell} \, \mathrm{Re}( \alpha_{\ell}(z_{j}) \overline{\beta_{\ell}(z_{j})}) > 1 - \frac{(C'\delta_{j})^{2}}{2},
\end{equation*}
and finally 
\begin{equation*}
    \textrm{Re}( \alpha_{0}(z_{j}) \overline{\beta_{0}(z_{j})}) > \frac{1 -\frac{(C'\delta_{j})^{2}}{2} - \delta_{j}}{s_0 - \delta_{j}} > 
    \frac{1 -\frac{(C' \delta_{j})^{2}}{2} - \delta_{j}}{1- \frac{\delta_{j}^{2}}{2}- \delta_{j}} > 1 - \frac{(C' \delta_{j})^{2}}{2}
\end{equation*}
where we have used that $s_{0}< \sqrt{1-\delta_{j}^{2}} < 1- \frac{\delta_{j}^{2}}{2}$ and $\delta_{j} < \frac{1}{2C'^{2}}$. 
This shows $\|{\mathbf u}_{j+1} {\mathbf v}_{j+1}^{*} - {\mathbf z}_{M}(z_{j}){\mathbf z}_{N}(z_{j})^{T}\|_{F} \le C'\delta_{j}$.

4.\ Finally, we conclude that the limit ${\mathbf u}{\mathbf v}^{*}$ of the subsequence ${\mathbf u}_{j_{\ell}} {\mathbf v}_{j_{\ell}}^{*}$ also has to be in this ball around ${\mathbf z}_{M}(z_{j}){\mathbf z}_{N}(z_{j})^{T}$ with radius $C'\delta_{j}$ for any $j$, and thus
$$\| {\mathbf u}_{j} {\mathbf v}_{j}^{*}  - {\mathbf u} {\mathbf v}^{*} \|_{F} \le 2C'\delta_{j}.$$
Since $\lim\limits_{j\to \infty} \delta_{j} =0$ we therefore have convergence of ${\mathbf u}_{j} {\mathbf v}_{j}^{*}$ to ${\mathbf u}{\mathbf v}^{*}$.
\end{proof}

\begin{remark}\label{remlast}
Note that the results of \cite{Lewis} cannot be applied to show convergence of the Cadzow algorithm since the considered manifolds do not satisfy the transversality condition, which is necessary in those proofs.
In \cite{Andersson13}, the condition of transversality is relaxed and replaced by the weaker condition of existence of nontangential intersection points. 
However the convergence results in \cite{Andersson13} rely on the assumption  that the angle $\alpha$ between the considered manifolds is bounded away from $0$, or equivalently, that the value $\sigma({\mathbf A})$ in Definition 3.1 of \cite{Andersson13} is smaller than $1$, which is not simple to show in the considered setting, and possibly not satisfied.
\end{remark}

\begin{example} \label{Cadex}
We show in a special example that Cadzow's algorithm for rank-1 Hankel approximation may indeed converge to the zero matrix. We consider the matrix 
\begin{equation*}
    {\mathbf A} \coloneqq \begin{pmatrix}
                            1 & 0 & 1/2 \\
                            0 & 1/2 & 0 \\
                            1/2 & 0 & 1
                        \end{pmatrix}
\end{equation*}
with eigenvalues $\frac{3}{2}$, $\frac{1}{2}$, $\frac{1}{2}$.
The singular vector to the largest singular value $\frac{3}{2}$  is of the form 
   $ {\mathbf u}_{0} = {\mathbf v}_{0} = \frac{1}{\sqrt{2}} \, (1, \, 0, \, 1)^{T}$.
Thus we find
\begin{equation*}
    P\left({\mathbf u}_{0} \, {\mathbf v}_{0}^{*}\right)
    = \frac{1}{2} 
        \begin{pmatrix}
            1 & 0 & 2/3 \\
            0  & 2/3 & 0 \\
            2/3 & 0 & 1
        \end{pmatrix}
    = \begin{pmatrix}
        1/2 & 0 & 1/3 \\
        0 & 1/3 & 0 \\
        1/3 & 0 & 1/2
        \end{pmatrix}.
\end{equation*}
Now,  ${\mathbf u}_{1} = {\mathbf v}_{1} = \frac{1}{\sqrt{2}} ( 1 , \, 0 , \, 1)^{T}$ is the singular vector of $P({\mathbf u}_0 \, {\mathbf v}_0^{*})$ to the largest singular value $5/6$. Further iterations yield 
\begin{equation*}
    {\mathbf u}_{j} = {\mathbf v}_{j} = \frac{1}{\sqrt{2}} ( 1 , \, 0 , \, 1)^{T}, \qquad \sigma_{j} = \frac{3}{2} \cdot \left( \frac{5}{6} \right)^{j} .
\end{equation*}
Obviously, $({\mathbf u}_{j})_{j=0}^{\infty}$ and  $({\mathbf v}_{j})_{j=0}^{\infty}$ are constant sequences with limit vector $\frac{1}{\sqrt{2}} ( 1 , \, 0 , \, 1)^{T}$, and $\lim_{j \to \infty} \sigma_{j} = 0$. In other words, the Cadzow algorithm fails to converge to a rank-1 matrix.\\
For comparison, \cref{theo1} provides the optimal rank-1 Hankel approximation with regard to the Frobenius norm $\tilde{c} \, \tilde{\mathbf z} \, \tilde{\mathbf z}^{T}$ with 
\begin{equation*}
    \tilde{z} = \mathop{\mathrm{argmax}}_{z \in {\mathbb C} } |{\mathbf z}^{T} \, {\mathbf A} \,  {\mathbf z}|^{2}
    = \mathop{\mathrm{argmax}}_{z \in {\mathbb C} } \frac{1 + \frac{3}{2} z^{2} + z^{4}}{1 + z^{2} + z^{4}} = \mathop{\mathrm{argmax}}_{z \in {\mathbb C} } \frac{z^{2}}{1+z^{2}+ z^{4}}.
\end{equation*}
We obtain the two solutions $\tilde{z}=1$ and $\tilde{z} = -1$. For both, $\tilde{z}=1$ and $\tilde{z}=-1$, we find $\tilde{c} = \mathbf{\tilde{z}}^T \mathbf{A} \, \mathbf{\tilde{z}} = \frac{7}{6}$. Thus, we get indeed two optimal solutions, namely
\begin{equation*}
    \frac{7}{18} \begin{pmatrix} 1 & 1 & 1 \\ 1 & 1 & 1 \\ 1 & 1 & 1 \end{pmatrix} \qquad \text{and} \qquad 
    \frac{7}{18} \begin{pmatrix} 1 & -1 & 1 \\ -1 & 1 & -1 \\ 1 & -1 & 1 \end{pmatrix}.
\end{equation*}
Both solutions possess the error
\begin{equation*}
    \|{\mathbf A} - \tilde{c} \, \tilde{\mathbf z}\, \tilde{\mathbf z}^{T} \|_{F}
    = \left\| \frac{1}{18}
    \begin{pmatrix}  
        11 & \pm 7 & 2 \\
        \pm 7 & 2 & \pm 7 \\
        2 & \pm 7 & 11
    \end{pmatrix} \right\|_{F}=  \frac{\sqrt{450}}{18} = 1.178511.
\end{equation*}
The spectral norm for this error matrix is $\|{\mathbf A} - \tilde{c} \, \tilde{\mathbf z} \, \tilde{\mathbf z}^{T}\|_{2} = 1.045820$. \\
Finally, let us consider the optimal rank-1 Hankel approximation of ${\mathbf A}$ with respect to the spectral norm.
We observe that the eigenvectors of ${\mathbf A}$ corresponding to $\frac{3}{2}$, $\frac{1}{2}$, $\frac{1}{2}$ are 
\begin{equation*}
    {\mathbf v}_{0} = \frac{1}{\sqrt{2}} ( 1 , \, 0 , \, 1)^{T}, \quad 
    {\mathbf v}_{1} =  ( 0 , \, 1 , \, 0)^{T}, \quad 
    {\mathbf v}_{2} = \frac{1}{\sqrt{2}} ( 1 , \, 0 , \, -1)^{T}.
\end{equation*}
The optimal error $\tilde{\lambda}$ is in the interval $[\frac{1}{2}, \frac{3}{2})$. Since $v_{1}(z) = {\mathbf v}_{1}^{T} {\mathbf z}$ and $v_{2}(z) = {\mathbf v}_{2}^{T} {\mathbf z}$ have no common zeros, 
we obtain $\Sigma= \emptyset$ in \cref{alg1}. We need to find $\tilde{\lambda}^{2}$ and $\tilde{z}$,  such that $f(\tilde{z}, \tilde{\lambda}^{2})$ satisfies \cref{spectral2}, i.e.,
$\max_{z \in {\mathbb R}} f_{\tilde{\lambda}}(z) =0$ and $\tilde{z} \in \mathop{\mathrm{argmax}}_{z \in {\mathbb R}} f_{\tilde{\lambda}}(z)$. We obtain
$$  f_{\lambda}(z)
    = \frac{({\mathbf v}_{0}^{T} {\mathbf z})^{2}}{\frac{9}{4} - \lambda^{2}} + \frac{({\mathbf v}_{1}^{T} {\mathbf z})^{2}}{\frac{1}{4} - \lambda^{2}} + \frac{({\mathbf v}_{2}^{T} {\mathbf z})^{2}}{\frac{1}{4} - \lambda^{2}}, 
$$
i.e.,    
\begin{align*}
 (1 + z^{2} + z^{4})    f_{\lambda}(z)
    &= \frac{\frac{1}{2} + z^{2} + \frac{z^{4}}{2}}{\frac{9}{4}- \lambda^{2}} + \frac{z^{2}}{\frac{1}{4} - \lambda^{2}} + \frac{\frac{1}{2} - z^{2} + \frac{z^{4}}{2}}{\frac{1}{4}- \lambda^{2}} \\
    &= \frac{1}{\left(\frac{9}{4}-\lambda^{2}\right)\left(\frac{1}{4}-\lambda^{2}\right)} \left( \left(\frac{5}{4}- \lambda^{2}\right) z^{4} + \left(\frac{1}{4} - \lambda^{2}\right) z^{2} + \left(\frac{5}{4}- \lambda^{2}\right) \right) \\
    &= \frac{\frac{5}{4} - \lambda^{2}}{\left(\frac{9}{4}-\lambda^{2}\right)\left(\frac{1}{4}-\lambda^{2}\right)}
    \left( \!\!\left( z^{2}- \left( \frac{\lambda^{2} -\frac{1}{4}}{2\left(\frac{5}{4}-\lambda^{2}\right)}\right)\!\!\right)^{2} \!\!+
   \!\! 1 - \!\!\left( \frac{\lambda^{2} -\frac{1}{4}}{2\left(\frac{5}{4}-\lambda^{2}\right)}\right)^{2} \right), 
\end{align*}
where we assume in the last line that $\lambda^2 \neq \frac{5}{4}$.
A direct inspection of $f_{\lambda}(z)$ provides that $\max_z f_{\tilde{\lambda}} (z) =0$ if and only if 
\begin{equation*}
    1 - \left( \frac{\tilde{\lambda}^{2}-\frac{1}{4}}{2\left(\frac{5}{4}-\tilde{\lambda}^{2}\right)}\right)^{2} =0,
\end{equation*}
i.e., if $\tilde{\lambda}^{2} = \frac{11}{12}$. We thus obtain from \cref{spectral2} and \cref{spectral3}
\begin{equation*}
    \tilde{z}^{2} =  \frac{\tilde{\lambda}^{2}-\frac{1}{4}}{2\left(\frac{5}{4}-\tilde{\lambda}^{2}\right)} = \frac{\frac{11}{12}-\frac{1}{4}}{2\left(\frac{5}{4} -\frac{11}{12}\right)} = 1, \qquad \tilde{c}= \left(\frac{\frac{2}{3}}{\frac{3}{2}-\sqrt{\frac{11}{12}}} + \frac{\frac{1}{3}}{\frac{1}{2}-\sqrt{\frac{11}{12}}} + 0\right)^{-1} = 2,
\end{equation*}
and therefore again the two solutions $\tilde{z}=1$ and $\tilde{z} = -1$.
For the obtained error matrix we have
\begin{equation*}
    \left\| {\mathbf A} - \frac{2}{3} \begin{pmatrix}
                                        1 & 1 & 1 \\
                                        1 & 1 & 1 \\
                                        1 & 1 & 1
                                    \end{pmatrix} \right\|_{2}
    =  \left\| \frac{1}{6} \begin{pmatrix}
                                2 & \pm 4 & -1 \\
                                \pm 4 & -1 & \pm 4 \\
                                -1 &  \pm 4 & 2
                            \end{pmatrix} \right\|_{2} = \sqrt{\frac{11}{12}} = 0.957427,
\end{equation*}
while for the Frobenius norm we get $\|{\mathbf A} - \tilde{c} \, \tilde{\mathbf z} \, \tilde{\mathbf z}^{T}\|_{F} = 1.443376$. 
By construction, the error matrix $ {\mathbf A} - \tilde{c} \, \tilde{\mathbf z} \, \tilde{\mathbf z}^{T}$ possesses the eigenvalues $\sqrt{\frac{11}{12}}$, $-\sqrt{\frac{11}{12}}$, and $\frac{1}{2}$.
\end{example}

\begin{example} \label{ex2}
Finally, we consider Example 5 in \cite{Gillard11}. Given the matrix 
\begin{equation*}
    {\mathbf A}_{a} = \begin{pmatrix} a & 1 & a & 1 & a \\ 1 & a & 1 & a & 1 \end{pmatrix}^{T},
\end{equation*}
we obtain for $a=0$ with the Cadzow algorithm and with the optimal Frobenius approximation in Section 3, respectively, 
\begin{align*}
    {\mathbf H}_{Cadzow} &= \begin{pmatrix} 0 & 0 & 0 & 0 & 0 \\ 0 & 0 & 0 & 0 & 1 \end{pmatrix}^{T},\\
    {\mathbf H}_{Frob} &= \begin{pmatrix} \pm 0.4469 & 0.4670 &  \pm 0.4881 & 0.5101 & \pm 0.5331 \\
    0.4670 &  \pm 0.4881 & 0.5101 & \pm 0.5331 & 0.5571 \end{pmatrix}^{T}. 
\end{align*}
For the Frobenius norm, we find two optimal solutions, $(\tilde{z}, \tilde{c}) = (1.045082,  \, 0.446855)$, and 
$(\tilde{z}, \tilde{c}) = (-1.045082, \, -0.446855)$, producing the same optimal error. We obtain 
\begin{equation*}
    \|{\mathbf A} - \tilde{c} \, \tilde{\mathbf z} \, \tilde{\mathbf z}^{T}\|_{F} = 1.577594, \qquad \|{\mathbf A} - {\mathbf H}_{Cadzow}\|_{F} = 2.
\end{equation*}
Note that the two algorithms HSVD and HTLS studied for comparison in \cite{Gillard11}, completely fail in this case.
For $a=2$ we get
\begin{align*}
    {\mathbf H}_{Cadzow} &= \begin{pmatrix}
                                1.5629 & 1.5369 & 1.5113 & 1.4861 & 1.4614 \\
                                1.5369 & 1.5113 & 1.4861 & 1.4614 & 1.4370
                            \end{pmatrix}^{T},\\
    {\mathbf H}_{Frob} &= \begin{pmatrix}
                            1.5563 & 1.5334 &  1.5108 & 1.4885 & 1.4666 \\
                            1.5334 &  1.5108 & 1.4885 & 1.4666 & 1.4450
                          \end{pmatrix}^{T}. 
\end{align*}
For the Frobenius norm, we have the solution $(\tilde{z}, \tilde{c}) = (0.985274, \, 1.556291)$.
We obtain 
\begin{equation*}
    \|{\mathbf A} - \tilde{c} \, \tilde{\mathbf z} \, \tilde{\mathbf z}^{T}\|_{F} = 1.577618,  \qquad \|{\mathbf A} - {\mathbf H}_{Cadzow}\|_{F} = 1.577681.
\end{equation*}
While for  $a=0$, Cadzows algorithms provides a solution error which is significantly larger than the optimal error, we get for $a=2$ an error which is almost optimal.
\end{example}

\section*{Conclusion and Outlook}

In \cref{sec:hankel_mat} we showed that a rank-1 Hankel matrix $\mathbf{H}_1$ is always of the form
$  \mathbf{H}_1 = c \, {\mathbf z}_{M} \, {\mathbf z}_{N}^{T}$ 
or $\mathbf{H}_1 = c \, {\mathbf e}_{M} \, {\mathbf e}_{N}^{T}$
with ${\mathbf z}_{N}$ and ${\mathbf e}_{N}$ defined in \cref{z} and \cref{en}. This observation enabled us to analytically solve
\begin{equation*}
  \min_{\mathbf{H}_{1}\in\mathbb{C}^{M\times N}} \left\|{\mathbf A}- {\mathbf H}_{1}\right\|^2_F \quad \text{and}  \quad \min_{\mathbf{H}_{1}\in\mathbb{C}^{M\times N}} \left\|{\mathbf A}-  \mathbf{H}_{1}\right\|^{2}_{2}.
\end{equation*}
In the case of the Frobenius norm our results apply to general matrices ${\mathbf A} \in \mathbb{C}^{M\times N}$. For the spectral norm we considered real symmetric matrices. Our theoretical results gave rise to algorithms to compute the optimal rank-1 Hankel approximations for the Frobenius and spectral norm. In particular,  the optimal solutions for the two norms usually differ. This is in contrast to well-known results for unstructured optimal low-rank approximations.\\
We showed that the well-known Cadzow algorithm applied for rank-1 Hankel approximation always converges to a fixed point. However, it can happen that the algorithm converges to the zero matrix. Even if Cadzow's method converges to a rank-1 Hankel matrix it usually does not converge to the optimal solution, neither with respect to the Frobenius norm nor with respect to the spectral norm. We conjecture that the fixed point reached by the Cadzow algorithm coincides with the optimal rank-1 Hankel approximation with respect to the Frobenius or spectral norm only in the trivial case, if the unstructured rank-1 approximation obtained by the singular value decomposition already has the wanted Hankel structure. In this case, Cadzow's algorithm stops already after one iteration step. \\
A natural extension of our results would be to ask for analytic solutions to the approximation problem for  Hankel matrices with rank $r >1$. However, due to an increasing number of special cases regarding the structure of higher-rank Hankel matrices, this problem is much more difficult to solve.
For the Frobenius norm, we will consider applying our algorithm iteratively in order to get a Hankel approximation of higher rank and study the obtained results in comparison to other numerical methods for low-rank Hankel approximation.

\subsection*{Acknowledgement}
The authors would like to thank Ingeborg Keller for helpful remarks to improve this manuscript. The authors owe profound thanks to an anonymous referee who pointed out several possibilities to considerably increase the quality of this work.
Support by the German Research Foundation in the framework of the RTG 2088 is gratefully acknowledged.

\small{
\bibliography{myBib}{}
\bibliographystyle{plain}
}

\end{document}